\documentclass[12pt, letterpaper]{amsart}
\ifdefined\pdfminorversion
\fi
\ifdefined\pdfobjcompresslevel
\fi
\usepackage[a4paper,margin=1in]{geometry}
\usepackage[T1]{fontenc}
\usepackage{microtype}
\usepackage{amsmath,amssymb,amsthm,mathtools}
\usepackage{mathrsfs}
\usepackage{tikz-cd}
\usepackage{xcolor}
\usepackage{xurl}
\usepackage{hyperref}
\hypersetup{colorlinks=true,linkcolor=blue!45!black,citecolor=blue!45!black,urlcolor=blue!45!black}
\newtheorem{theorem}{Theorem}[section]
\newtheorem{proposition}[theorem]{Proposition}
\newtheorem{lemma}[theorem]{Lemma}
\newtheorem{corollary}[theorem]{Corollary}
\newtheorem{conjecture}[theorem]{Conjecture}
\theoremstyle{definition}
\newtheorem{definition}[theorem]{Definition}
\theoremstyle{remark}
\newtheorem{remark}[theorem]{Remark}
\newcommand{\Q}{\mathbb Q}
\newcommand{\C}{\mathbb C}
\newcommand{\Z}{\mathbb Z}
\newcommand{\Hom}{\operatorname{Hom}}
\newcommand{\End}{\operatorname{End}}
\DeclareMathOperator{\ord}{ord}
\numberwithin{equation}{section}

\title{Strong Weil Degree Divisibility at Higher Levels}

\author{Daeyeol Jeon}
\address{Department of Mathematics Education, Kongju National University, Gongju 32588, Republic of Korea}
\email{dyjeon@kongju.ac.kr}

\author{Yongjae Kwon}
\address{Department of Mathematics Education, Kongju National University, Gongju 32588, Republic of Korea}
\email{211049@kongju.ac.kr}
\email{kwon314159@gmail.com}


\date{}

\begin{document}

\begin{abstract}
Let \(\pi_E:X_0(M)\to E\) be the strong Weil parametrization with Manin
constant \(c_E\). We prove \(\deg\pi_E\mid c_E^{\Omega(N/M)}\deg g\) for
every multiple \(N\) of \(M\) and every nonconstant morphism
\(g:X_0(N)\to E'\) over \(\Q\), where \(E'\) is \(\Q\)-isogenous to \(E\)
and \(\Omega\) counts prime factors with multiplicity.
When \(c_E=1\), as is known for squarefree \(M\), the modular degree at
level \(M\) therefore divides every such degree at every higher level.
As an application of the divisibility theorem, we prove that no
\(X_0(N)/\Q\) admits a morphism over \(\Q\) of positive odd degree at most
\(1645\) to an elliptic curve of positive \(\Q\)-rank.
For a fixed target \(E'\) and a generator \(u:E\to E'\), we also prove that
if the Manin constant \(c_{u\circ\pi_E}=1\), the old homomorphisms induced by
degeneracy maps form an integral basis of \(\Hom_\Q(J_0(N),E')\), and the
old degree matrix determines the exact morphism degrees.
The proofs bound denominators in the rational old basis.
The divisibility and lattice results extend to compatible towers of
intermediate modular curves, including the \(X_1\)-tower.
\end{abstract}

\maketitle

\section{Introduction}
Let \(E/\Q\) be the \(X_0(M)\)-optimal curve, also called the strong Weil curve,
of conductor \(M\). Let \(j_M:X_0(M)\to J_0(M)\) be the Abel--Jacobi map based
at the cusp \(\infty\). Write \(\pi_E:X_0(M)\to E\) and
\(\Phi_E:J_0(M)\to E\) for the strong Weil parametrization and the induced optimal quotient of
Jacobians, so that \(\pi_E=\Phi_E\circ j_M\). At the minimal level \(M\), the
universal property of the optimal \(E\)-isogenous quotient
\cite{DerickxOrlic} gives an integral
description of the Hom group. If \(E'/\Q\) is in the
\(\Q\)-isogeny class of \(E\), then composition with \(\Phi_E\) induces an
isomorphism of \(\Z\)-modules
\begin{equation}\label{eq:intro-level-M-hom}
        \Hom_\Q(E,E')
        \xrightarrow{\ \sim\ }
        \Hom_\Q(J_0(M),E'),
        \qquad
        v\longmapsto v\circ\Phi_E .
\end{equation}
Thus, if \(u:E\to E'\) generates the free rank-one
\(\Z\)-module \(\Hom_\Q(E,E')\),
then \(u\circ\Phi_E\) is a \(\Z\)-basis of
\(\Hom_\Q(J_0(M),E')\).
Consequently, every nonconstant morphism \(g:X_0(M)\to E'\) over
\(\Q\) factors through \(\pi_E\) after translation, and
\(\deg\pi_E\mid\deg g\).

The level-\(M\) argument suggests two questions for higher levels. Let \(N\) be a
multiple of \(M\), and put \(R=N/M\). For \(r\mid R\), let
\(\iota_{r,N,M}:X_0(N)\to X_0(M)\) be the degeneracy map analytically induced by
\(\tau\mapsto r\tau\). For a fixed elliptic curve \(E'/\Q\) in the
\(\Q\)-isogeny class of \(E\) and a nonzero isogeny \(u:E\to E'\), define
\(\Phi_{r,u}=u\circ\Phi_E\circ\iota_{r,N,M,*}:J_0(N)\to E'\).
The first question is whether the divisibility \(\deg\pi_E\mid\deg g\) persists
for nonconstant morphisms \(g:X_0(N)\to E'\). The second, fixed-target, question
asks whether the old homomorphisms \(\{\Phi_{r,u}:r\mid R\}\) form an integral
basis of \(\Hom_\Q(J_0(N),E')\), or at least whether the denominators of the old
lattice they generate can be bounded explicitly.

The main results answer both questions up to explicit factors controlled by
Manin constants. When \(c_E=1\), the divisibility
\(\deg\pi_E\mid\deg g\) holds. When \(c_{u\circ\pi_E}=1\), the old maps form
a \(\Z\)-basis of the full Hom group.

Let \(f_E\) be the normalized rational newform attached to \(E\), and let
\(c_E\) be the Manin constant of \(\pi_E\). Choose the sign of a minimal
N\'eron differential \(\omega_E\) on \(E\) so that
\(\pi_E^*\omega_E=c_E f_E(q)dq/q\), with \(c_E>0\).
For a positive integer \(n\), put \(\Omega(n)=\sum_p\ord_p(n)\), with
\(\Omega(1)=0\).

\begin{theorem}\label{thm:main}
Let \(E/\Q\) be the strong Weil curve of conductor \(M\), and let
\(\pi_E:X_0(M)\to E\) be the strong Weil parametrization. Let \(c_E>0\) be the
Manin constant of \(\pi_E\). Then, for every multiple \(N\) of \(M\), every
elliptic curve \(E'\) over \(\Q\) in the \(\Q\)-isogeny class of \(E\), and every
nonconstant morphism \(g:X_0(N)\to E'\) over \(\Q\), one has
\[
        \deg\pi_E\mid c_E^{\Omega(N/M)}\deg g.
\]
In particular, if \(c_E=1\), then \(\deg\pi_E\mid\deg g\).
\end{theorem}

A sharper divisibility statement for \(N>M\) is given in
Corollary~\ref{cor:x0-sharp-specialization}.
By \v{C}esnavi\v{c}ius \cite{CesnaviciusManinSemistable},
\(c_E=1\) when \(M\) is squarefree. Cremona's tables
\cite{CremonaData} verify \(c_E=1\) for \(M<400000\).
The divisibility \(\deg\pi_E\mid\deg g\) is therefore unconditional in these cases.

For a nonzero isogeny \(u:E\to E'\), write \(c_{u\circ\pi_E}\) for the positive
integer determined by
\((u\circ\pi_E)^*\omega_{E'}=c_{u\circ\pi_E}f_E(q)dq/q\), where
\(\omega_{E'}\) is a minimal N\'eron differential. If
\(u^*\omega_{E'}=c_u\omega_E\), then
\(c_{u\circ\pi_E}=c_Ec_u\) by Lemma \ref{lem:fixed-target-manin}. In particular,
when \(E'=E\) and \(u=\mathrm{id}_E\), the condition
\(c_{u\circ\pi_E}=1\) is precisely \(c_E=1\). When
\(c_{u\circ\pi_E}=1\), the level-\(M\) factorization extends integrally to every
higher level.

\begin{theorem}\label{thm:lattice-main}
With the notation above, if
\(u\) generates \(\Hom_\Q(E,E')\) and \(c_{u\circ\pi_E}=1\), then
\(\{\Phi_{r,u}:r\mid R\}\) is a \(\Z\)-basis of
\(\Hom_\Q(J_0(N),E')\).
\end{theorem}

Equivalently, under the hypothesis \(c_{u\circ\pi_E}=1\), every
\(H\in\Hom_\Q(J_0(N),E')\) factors uniquely through \(u\). More precisely,
\(H=u\circ H_E\) for
\(H_E\in \bigoplus_{r\mid R}\Z\,\Phi_{r,\mathrm{id}_E}\). Here
\(\Phi_{r,\mathrm{id}_E}=\Phi_E\circ\iota_{r,N,M,*}:J_0(N)\to E\).
Thus, for every morphism \(g:X_0(N)\to E'\), there is a morphism
\(h:X_0(N)\to E\) over \(\Q\) satisfying \(h(\infty)=0\) and
\(t_{-g(\infty)}\circ g=u\circ h\).

For \(d\)-elliptic curves in the sense of
\cite{DerickxOrlic}, the basis in Theorem
\ref{thm:lattice-main} identifies the full integral Hom lattice with the old
degeneracy lattice. The Derickx--Orli\'c degree pairing
\cite{DerickxOrlic} is positive
definite and has the degree of a morphism as its diagonal value. Hence the old
degree matrix turns the possible degrees into the values represented by an
integral quadratic form.

We also extend Theorems \ref{thm:main} and \ref{thm:lattice-main} to compatible
towers of intermediate modular curves between \(X_1\) and \(X_0\). See Theorems
\ref{thm:intermediate-tower} and \ref{thm:intermediate-optimal}.
In particular, let \(\pi_{E_1}:X_1(M)\to E_1\) be an \(X_1(M)\)-optimal
parametrization with Manin constant \(c_1\). Theorems \ref{thm:x1-final} and
\ref{thm:x1-fixed-target} show
that, for every multiple \(N\) of \(M\) and every nonconstant morphism
\(g:X_1(N)\to E'\) to an elliptic curve \(\Q\)-isogenous to \(E_1\),
\(\deg\pi_{E_1}\mid c_1^{\Omega(N/M)}\deg g\). For a fixed target \(E'\) and a
generator \(u:E_1\to E'\), if the Manin constant \(c_{u\circ\pi_{E_1}}\) is
\(1\), then the old homomorphisms induced by the degeneracy maps form a
\(\Z\)-basis of \(\Hom_\Q(J_1(N),E')\).
For squarefree \(M\), we prove unconditionally that the old homomorphisms
form a \(\Z[1/2]\)-basis of
\(\Hom_\Q(J_1(N),E')\otimes_\Z\Z[1/2]\) for every multiple \(N\) of \(M\)
and every target \(E'\) in the \(\Q\)-isogeny class of \(E_1\)
(Theorem~\ref{cor:x1-vatsal-lattice}).

Combining Theorem~\ref{thm:main} with modular-degree bounds
gives the following unconditional application to Mordell--Weil ranks.

\begin{theorem}\label{thm:intro-rank-applications}
For every positive odd integer \(d\leq1645\), no modular curve
\(X_0(N)/\Q\) admits a degree-\(d\) morphism over \(\Q\) to an elliptic curve
of positive \(\Q\)-rank.
\end{theorem}

Section~\ref{sec:DO-setup} reviews the background, and
Section~\ref{sec:old-degree-form} constructs the rational old basis and its
degree matrix. In Section~\ref{sec:denominators}, two-cusp estimates and
connected-kernel descent establish the denominator bound, which implies
Theorem~\ref{thm:lattice-main}. Applying this bound after composition with a
dual isogeny proves Theorem~\ref{thm:main} via the degree matrix.
Subsection~\ref{subsec:rank-consequences} contains the applications to
Mordell--Weil ranks. Section~\ref{sec:tower-intermediate} extends the arguments
to compatible intermediate modular curves and \(X_1\), using congruence
subgroups to establish connected-kernel descent.

\section{Notation and background}\label{sec:DO-setup}
We use the notation from the preliminary sections of
Derickx--Orli\'c \cite{DerickxOrlic}, together with the
integral structures on differentials used later. A curve over a field is assumed to
be smooth, projective, and geometrically integral. Unless another field is
specified, morphisms and homomorphisms are defined over \(\Q\).

\subsection{Jacobians and degree pairings}

Let \(C\) be a curve over a field \(k\). Write
\(J(C)=\operatorname{Pic}^0_{C/k}\) for its Jacobian. If
\(f:C\to C'\) is a nonconstant morphism of curves, then pushforward and pullback
of divisor classes induce homomorphisms \(f_*:J(C)\to J(C')\) and
\(f^*:J(C')\to J(C)\).
The identity
\begin{equation}\label{eq:push-pull-degree}
        f_*\circ f^*=[\deg f]
\end{equation}
holds on \(J(C')\) \cite{DerickxOrlic}. If \(P\in C(k)\), write
\(j_P:C\to J(C)\), \(x\mapsto [x-P]\), for the Abel--Jacobi map. For an abelian variety \(A/k\), the universal property
of $J(C)$ identifies \(\Hom_k(J(C),A)\) with pointed morphisms
\cite{MilneJacobians}.
\begin{equation}\label{eq:jacobian-universal-property}
        \Hom_k(J(C),A)\simeq
        \{h:C\to A\mid h(P)=0_A\},
        \qquad H\longmapsto H\circ j_P.
\end{equation}
By \eqref{eq:jacobian-universal-property}, translating a morphism
\(h:C\to A\) by \(-h(P)\) produces the pointed map represented by a
homomorphism from \(J(C)\). If \(A\) is an elliptic curve,
translation does not change the degree of a nonconstant morphism or the pullback
of an invariant differential.
For abelian varieties \(A,B/\Q\), the notation \(\Hom_\Q(A,B)\) denotes the
abelian group of homomorphisms defined over \(\Q\). Rational linear combinations
of such homomorphisms are taken in \(\Hom_\Q(A,B)\otimes_\Z\Q\).
For every elliptic curve \(E/\Q\), one has \(\End_\Q(E)=\Z\). Indeed, a
\(\Q\)-endomorphism acts faithfully on the one-dimensional \(\Q\)-space of
invariant differentials. Its scalar is an algebraic integer and hence an integer.
See \cite{SilvermanAEC}.

\begin{definition}\label{def:degree-pairing}
Let \(C/k\) be a curve and let \(E/k\) be an elliptic curve. For morphisms
\(f,g:C\to E\), the \emph{degree pairing} is
\begin{equation}\label{eq:degree-pairing-curve}
        \langle f,g\rangle=f_*\circ g^*\in \End_k(J(E)).
\end{equation}
After identifying \(J(E)\) with \(E\) by the origin, the endomorphism in
\eqref{eq:degree-pairing-curve} lies in \(\End_k(E)\). The push-pull identity
\eqref{eq:push-pull-degree} yields
\begin{equation}\label{eq:degree-pairing-diagonal}
        \langle f,f\rangle=[\deg f].
\end{equation}
When \(k=\Q\), the equality \(\End_\Q(E)=\Z\) is used to regard
\(\langle f,g\rangle\) as an integer. If \(P\in C(k)\), the same notation is used
on \(\Hom_k(J(C),E)\) by setting
\begin{equation}\label{eq:degree-pairing-jacobian}
        \langle H,K\rangle=\langle H\circ j_P,K\circ j_P\rangle .
\end{equation}
Equivalently, the pairing in \eqref{eq:degree-pairing-jacobian} is the Rosati,
or dagger, pairing attached to the canonical principal polarizations of \(J(C)\)
and \(J(E)\) \cite{DerickxOrlic}.
\end{definition}

\begin{proposition}\label{prop:degree-pairing-basic}
Let \(C/\Q\) be a curve with a rational point and let \(E/\Q\) be an elliptic
curve. The degree pairing is positive definite on \(\Hom_\Q(J(C),E)\). If
\(H\in\Hom_\Q(J(C),E)\) corresponds, up to translation of the codomain, to the
curve map \(h:C\to E\), then \(\langle H,H\rangle=\deg h\).
In particular, a \(\Z\)-basis of a subgroup of \(\Hom_\Q(J(C),E)\) determines an
integral quadratic form representing the degrees of the corresponding maps from
\(C\) to \(E\).
\end{proposition}

\begin{proof}
The assertions are the degree-pairing facts proved in
\cite{DerickxOrlic}. The
diagonal equality is \eqref{eq:degree-pairing-diagonal}, and translations of
the codomain do not change degree.
\end{proof}

\subsection{Optimal quotients and degeneracy maps}

We begin with the universal property of the optimal quotient attached to a
fixed simple isogeny factor.

\begin{definition}\label{def:optimal-isogenous}
Let \(A\) be an abelian variety over \(k\), and let \(B\) be a simple abelian
variety over \(k\). A quotient \(\varpi:A\twoheadrightarrow A_B\) is an
\emph{optimal \(B\)-isogenous quotient} if \(A_B\) is isogenous to \(B^n\) for
some \(n\ge0\) and, for every homomorphism \(\psi:A\to A'\) with \(A'\) isogenous
to a power of \(B\), there is a unique \(\bar\psi:A_B\to A'\) such that
\(\psi=\bar\psi\circ\varpi\).
Equivalently, the triangle
\[
\begin{tikzcd}[column sep=large,row sep=large]
A \arrow[r,"\psi"] \arrow[d,"\varpi"'] & A' \\
A_B \arrow[ur,dashed,"\bar\psi"']
\end{tikzcd}
\]
commutes. Dually, a \emph{maximal \(B\)-isogenous subvariety} of \(A\) is
universal for homomorphisms from abelian varieties isogenous to powers of \(B\)
into \(A\). We use the terms optimal quotient and maximal subvariety as in
Derickx--Orli\'c \cite{DerickxOrlic}.
\end{definition}

An elliptic curve \(E/\Q\) of conductor \(M\) is called the \emph{strong Weil
curve}, or the \emph{\(X_0(M)\)-optimal curve}, in its isogeny class if
\(\Phi_E:J_0(M)\twoheadrightarrow E\) is the optimal \(E\)-isogenous quotient.
Equivalently, \(\ker\Phi_E\) is geometrically connected and has no simple quotient
isogenous to \(E\) \cite{AgasheRibetStein}. The associated modular parametrization
is \(\pi_E:X_0(M)\xrightarrow{j_M}J_0(M)\xrightarrow{\Phi_E}E\).

Throughout the paper, an \emph{old parametrization} means a modular
parametrization composed with a degeneracy map.

\begin{definition}\label{def:degeneracy}
Let \(M\mid N\) and let \(r\mid N/M\). The \emph{degeneracy map}
\(\iota_{r,N,M}:X_0(N)\to X_0(M)\) is the algebraic morphism whose analytification is induced by
\(\tau\mapsto r\tau\) on the extended upper half-plane. On Jacobians it induces
\(\iota_{r,N,M,*}:J_0(N)\to J_0(M)\) and
\(\iota_{r,N,M}^{*}:J_0(M)\to J_0(N)\).
If $M\mid N\mid N'$ and the indices are defined, the degeneracy maps satisfy
\(\iota_{r,N,M}\circ \iota_{e,N',N}=\iota_{re,N',M}\).
When the source and target levels have been fixed, the notation is abbreviated to
\(\iota_r\), \(\iota_{r,*}\), and \(\iota_r^*\).
\end{definition}

Let \(E/\Q\) be the strong Weil curve of conductor \(M\), let \(E'/\Q\) be in
its \(\Q\)-isogeny class, and let \(u:E\to E'\) be a nonzero isogeny.
For every level \(L\), write \(j_L:X_0(L)\to J_0(L)\) for the
Abel--Jacobi map based at the rational cusp \(\infty\).
Since the degeneracy maps send \(\infty\) to \(\infty\), the old degeneracy
parametrizations and the old homomorphisms are compatible with Abel--Jacobi, by
functoriality of Jacobians \cite{MilneJacobians}. There is therefore a commutative
diagram
\[
\begin{tikzcd}[column sep=large,row sep=large]
X_0(N) \arrow[r,"\iota_r"] \arrow[d,"j_N"'] &
X_0(M) \arrow[r,"\pi_E"] \arrow[d,"j_M"'] &
E \arrow[r,"u"] \arrow[d,equal] &
E' \arrow[d,equal] \\
J_0(N) \arrow[r,"\iota_{r,*}"'] &
J_0(M) \arrow[r,"\Phi_E"'] &
E \arrow[r,"u"'] &
E' .
\end{tikzcd}
\]
Each old degeneracy parametrization has two equivalent descriptions, namely the
curve map \(u\circ\pi_E\circ\iota_r\) and the homomorphism
\(u\circ\Phi_E\circ\iota_{r,*}\) followed by \(j_N\). The convention is used
throughout the degree and denominator arguments.

\subsection{Differentials and the Manin constant}

Let \(f_E(q)=\sum_{m\ge1}a_mq^m\), with \(a_1=1\), be the normalized rational
newform attached to the conductor-\(M\) elliptic curve \(E\). A minimal
N\'eron differential on \(E\) means a generator, up to sign, of the free
rank-one \(\Z\)-module of invariant differentials on the N\'eron model of \(E\)
over \(\Z\).

\begin{definition}\label{def:manin-constant}
Let \(\omega_E\) be a minimal N\'eron differential. The \emph{Manin constant} of
\(\pi_E\) is the positive integer \(c_E\) determined by
\begin{equation}\label{eq:manin-constant-X0}
        \pi_E^*\omega_E=c_E f_E(q)\frac{dq}{q},
\end{equation}
where the sign of \(\omega_E\) is chosen with \(c_E>0\). The convention is
standard for Manin constants \cite{AgasheRibetStein,CesnaviciusNeururerSaha}.
\end{definition}

\begin{lemma}\label{lem:fixed-target-manin}
Let \(E'\) be an elliptic curve over \(\Q\) in the \(\Q\)-isogeny class of
\(E\), and let \(u:E\to E'\) be a nonzero isogeny. For minimal N\'eron
differentials \(\omega_E\) and \(\omega_{E'}\), choose the sign of
\(\omega_{E'}\) with \(u^*\omega_{E'}=c_u\omega_E\) and \(c_u>0\).
Then \(c_u\in\Z\). The composite \(u\circ\pi_E\) satisfies
\((u\circ\pi_E)^*\omega_{E'}=c_{u\circ\pi_E}f_E(q)dq/q\), where
\(c_{u\circ\pi_E}=c_Ec_u\in\Z_{>0}\).
If \(\widehat u:E'\to E\) is the dual isogeny and
\(\widehat u^*\omega_E=c_{\widehat u}\omega_{E'}\), with
\(c_{\widehat u}>0\), then \(c_{\widehat u}\in\Z\) and
\begin{equation}\label{eq:cu-divides-degree-u}
        c_uc_{\widehat u}=\deg u.
\end{equation}
In particular, \(c_u\mid\deg u\). Moreover, \(\deg(u\circ\pi_E)=\deg\pi_E\deg u\).
\end{lemma}

\begin{proof}
The isogeny \(u\) extends to a morphism between the N\'eron models over
\(\Z\) by the N\'eron mapping property \cite{BLR}.
Therefore pullback by \(u\) sends invariant N\'eron differentials on \(E'\) to
invariant N\'eron differentials on \(E\), and the scalar \(c_u\) is an integer.
The same argument applied to \(\widehat u\) shows that
\(c_{\widehat u}\in\Z\). Since \(\widehat u\circ u=[\deg u]\), pulling back
\(\omega_E\) yields
\[
        \deg u\cdot\omega_E
        =
        u^*\widehat u^*\omega_E
        =
        c_{\widehat u}u^*\omega_{E'}
        =
        c_uc_{\widehat u}\omega_E,
\]
which proves \eqref{eq:cu-divides-degree-u}.
The formula for \((u\circ\pi_E)^*\omega_{E'}\) follows from
\eqref{eq:manin-constant-X0}. The degree formula is multiplicativity of degrees
for finite morphisms of curves.
\end{proof}

\begin{lemma}\label{lem:q-expansion-degeneracy}
Let \(M\mid N\), let \(r\mid N/M\), and use the degeneracy normalization of
Definition \ref{def:degeneracy}. Write \(\iota_r=\iota_{r,N,M}\). If \(q_M\)
and \(q_N\) are the analytic
\(q\)-parameters \(e^{2\pi i\tau}\) at \(\infty\) on \(X_0(M)\) and \(X_0(N)\),
then, on completed local rings, \(q_M\) pulls back to \(q_N^r\). For a weight-two
modular form \(f(q)\), one has
\(\iota_r^{*}(f(q_M)dq_M/q_M)=f(q_N^r)d(q_N^r)/q_N^r\).
\end{lemma}

\begin{proof}
Analytically the map is induced by \(\tau\mapsto r\tau\), so
\(q_M=e^{2\pi i\tau}\) pulls back to \(e^{2\pi i r\tau}=q_N^r\). Pulling back
\(dq_M/q_M\) yields \(d(q_N^r)/q_N^r\). The association
\(f(q)dq/q\) between weight-two modular forms and regular differentials follows
the convention of \cite{Stein}.
\end{proof}

\section{Old degree matrices}\label{sec:old-degree-form}
Let \(C/\Q\) be a smooth projective geometrically connected curve with a rational
point \(P\), let \(E/\Q\) be an elliptic curve, and let
\(j_P:C\to J(C)\), \(x\mapsto [x-P]\), be the Abel--Jacobi map. For a morphism
\(h:C\to E\) over \(\Q\), put \(Q=h(P)\). The universal property
\eqref{eq:jacobian-universal-property} determines a unique homomorphism
\(H:J(C)\to E\) over \(\Q\) such that
\[
\begin{tikzcd}[column sep=large,row sep=large]
C \arrow[r,"j_P"] \arrow[dr,"{t_{-Q}\circ h}"'] &
J(C) \arrow[d,"H"] \\
& E
\end{tikzcd}
\]
commutes, where \(t_Q\) denotes translation by \(Q\). Equivalently,
\begin{equation}\label{eq:associated-jacobian-homomorphism}
        h=t_Q\circ H\circ j_P.
\end{equation}
We call \(H\) the homomorphism associated with \(h\), based at \(P\).

\begin{lemma}\label{lem:curve-jacobian}
Let \(h:C\to E\) and \(P\in C(\Q)\), and let
\(H:J(C)\to E\) be the homomorphism
associated with \(h\), based at \(P\), as in
\eqref{eq:associated-jacobian-homomorphism}. If \(h\) is nonconstant, then
\(\deg h=\deg(H\circ j_P)\). Moreover, for every invariant differential
\(\omega\) on \(E\), one has \(h^*\omega=(H\circ j_P)^*\omega\).
\end{lemma}

\begin{proof}
Translation on an elliptic curve is an isomorphism and acts trivially on
invariant differentials. Hence \eqref{eq:associated-jacobian-homomorphism} does
not change the degree of a nonconstant morphism or the pullback of \(\omega\).
\end{proof}

Let \(E/\Q\) be the strong Weil curve of conductor \(M\), and let \(E'/\Q\)
belong to its \(\Q\)-isogeny class. Let \(N\) be a multiple of \(M\), and let
\(u:E\to E'\) be a nonzero isogeny. For \(r\mid N/M\), write
\(\iota_r=\iota_{r,N,M}\) and define the \emph{old homomorphism}
\begin{equation}\label{eq:old-homomorphism-definition}
        \Phi_{r,u}=u\circ\Phi_E\circ\iota_{r,*}:J_0(N)\longrightarrow E'
        \qquad(r\mid N/M).
\end{equation}

The key structural point is that only the level-\(M\) newform component of
\(J_0(N)\) contributes to homomorphisms with target \(E'\). The old
homomorphisms arising from the degeneracy maps therefore span
\(\Hom_\Q(J_0(N),E')\otimes_\Z\Q\).

\begin{theorem}\label{thm:rational-old-basis}
Let \(E/\Q\) be the strong Weil curve of conductor \(M\), and let \(E'/\Q\)
belong to its \(\Q\)-isogeny class. Let \(N\) be a multiple of \(M\), and let
\(u:E\to E'\) be a nonzero isogeny over \(\Q\). Then
the old homomorphisms \(\{\Phi_{r,u}:r\mid N/M\}\) form a \(\Q\)-basis of
\(\Hom_\Q(J_0(N),E')\otimes_\Z\Q\).
\end{theorem}

\begin{proof}
Put \(V_N(E')=\Hom_\Q(J_0(N),E')\otimes_\Z\Q\).
For each \(L\mid N\), let \(V_L^{\mathrm{new}}(E')\) be the subspace of
\(\Hom_\Q(J_0(L),E')\otimes_\Z\Q\) consisting of homomorphisms that factor
through the new quotient of \(J_0(L)\).
The Atkin--Lehner--Li decomposition is
\begin{equation}\label{eq:ALL-forms}
        S_2(\Gamma_0(N))=
        \bigoplus_{L\mid N}\ \bigoplus_{r\mid N/L}
        \iota_{r,N,L}^{*}S_2(\Gamma_0(L))^{\mathrm{new}}.
\end{equation}
On Jacobians it induces an isogeny over \(\Q\)
\begin{equation}\label{eq:ALL-jacobians}
        \prod_{L\mid N}\ \prod_{r\mid N/L}J_0(L)^{\mathrm{new}}
        \longrightarrow J_0(N),
        \qquad
        (x_{L,r})\longmapsto
        \sum_{L,r}\iota_{r,N,L}^{*}(x_{L,r}).
\end{equation}
The decomposition \eqref{eq:ALL-forms} is the form decomposition, and
\eqref{eq:ALL-jacobians} is the induced statement up to isogeny on Jacobians
\cite{AtkinLehner,Li,Stein}. See also \cite{DiamondShurman} for the
oldform multiplicity statement. Dualizing \eqref{eq:ALL-jacobians} and using
the canonical principal polarizations of the ambient Jacobians identifies the
dual of each pullback with the corresponding pushforward
\cite{ConradEdixhovenStein}. After identifying each new factor with its
dual up to \(\Q\)-isogeny, applying
\(\Hom_\Q(-,E')\otimes_\Z\Q\) to the dual isogeny yields the direct-sum
decomposition
\[
        V_N(E')
        =
        \bigoplus_{L\mid N}\ \bigoplus_{r\mid N/L}
        \{\,h\circ\iota_{r,N,L,*}:h\in V_L^{\mathrm{new}}(E')\,\}.
\]

The summand with level \(L\) is zero unless \(J_0(L)^{\mathrm{new}}\) has a
simple quotient \(\Q\)-isogenous to \(E'\). By Faltings' isogeny theorem
\cite{Faltings}, together with the newform decomposition of modular Jacobians
\cite{Stein}, this can occur only for the newform attached to the
isogeny class of \(E'\). Since \(E'\) is \(\Q\)-isogenous to \(E\), this newform
has exact level \(M\). Hence only the terms with \(L=M\) remain.

At level \(M\), multiplicity one for the newform quotient
\cite{AtkinLehner,Li,Stein} identifies
\(V_M^{\mathrm{new}}(E')=\Q\cdot(u\circ\Phi_E)\).
Thus
\begin{equation}\label{eq:old-span-hom}
        \Hom_\Q(J_0(N),E')\otimes_\Z\Q
        =\bigoplus_{r\mid N/M}\Q\cdot
        \Phi_{r,u}.
\end{equation}
The direct sum in \eqref{eq:old-span-hom} proves both spanning and linear
independence.
\end{proof}

Thus every rational homomorphism at level \(N\) is old. The remaining questions
concern the integral structure. One must compare
\(\bigoplus_{r\mid N/M}\Z\Phi_{r,u}\) with
\(\Hom_\Q(J_0(N),E')\) and express degrees through the degree pairing in these
coordinates.

We use the following lattice terminology. Let \(L\) be a free abelian group of finite
rank. A subgroup \(L_0\subseteq L\) is \emph{saturated} if \(L/L_0\) is
torsion-free. A nonzero element \(h\in L\) is \emph{primitive} if \(\Z h\) is
saturated in \(L\). If \(L\) has rank one and \(h\) is primitive, then
\(L=\Z h\). In particular, \(qh\in L\) with \(q\in\Q\) implies \(q\in\Z\).

\begin{lemma}\label{lem:optimal-primitivity}
Let \(E/\Q\) be the strong Weil curve of conductor \(M\), and let \(E'\) be an
elliptic curve over \(\Q\) in the \(\Q\)-isogeny class of \(E\). Then composition
with \(\Phi_E\) induces the isomorphism
\(\Hom_\Q(E,E')\xrightarrow{\sim}\Hom_\Q(J_0(M),E')\),
\(v\mapsto v\circ\Phi_E\), of \(\Z\)-modules.
If \(u\) generates \(\Hom_\Q(E,E')\), then \(u\circ\Phi_E\) is primitive in
\(\Hom_\Q(J_0(M),E')\).
\end{lemma}

\begin{proof}
By the universal property of the optimal \(E\)-isogenous quotient
\cite{DerickxOrlic}, composition with \(\Phi_E\)
is an isomorphism of \(\Z\)-modules. Under this isomorphism, a generator \(u\)
of \(\Hom_\Q(E,E')\) maps to the generator \(u\circ\Phi_E\) of
\(\Hom_\Q(J_0(M),E')\), which is primitive.
\end{proof}

For \(r,s\mid N/M\), let \(a_{r,s}\in\Z\) be the unique integer characterized by
\begin{equation}\label{eq:old-degree-matrix-entry}
 \Phi_E\circ\iota_{r,*}\circ\iota_s^*=[a_{r,s}]\circ\Phi_E.
\end{equation}
The existence and uniqueness of \(a_{r,s}\) follow from Lemma
\ref{lem:optimal-primitivity} applied with \(E'=E\).
We define the \emph{old degree matrix} attached to \(E\) at level \(N\) by
\(A_{E,N}=(a_{r,s})_{r,s\mid N/M}\). Equivalently, \(a_{r,s}\) is characterized
by commutativity of the push--pull diagram.
\[
\begin{tikzcd}[column sep=large,row sep=large]
J_0(M) \arrow[r,"\iota_s^*"] \arrow[d,"\Phi_E"'] &
J_0(N) \arrow[r,"\iota_{r,*}"] &
J_0(M) \arrow[d,"\Phi_E"] \\
E \arrow[rr,"{[a_{r,s}]}"'] && E
\end{tikzcd}
\]
The entries \(a_{r,s}\) are also the normalized Gram matrix of the Derickx--Orli\'c degree
pairing on the old degeneracy parametrizations. By Definition
\ref{def:degree-pairing} and the equality \(\pi_{E,*}=\Phi_E\), one has
\begin{align}
 \langle \pi_E\circ\iota_r,\pi_E\circ\iota_s\rangle
 &=\Phi_E\circ\iota_{r,*}\circ\iota_s^{*}\circ\pi_E^*  \notag\\
 &=[a_{r,s}]\circ\Phi_E\circ\pi_E^*                       \notag\\
 &=[a_{r,s}\deg\pi_E].
        \label{eq:pairing-on-old-translates}
\end{align}
The second equality in \eqref{eq:pairing-on-old-translates} follows from
\eqref{eq:old-degree-matrix-entry}, while the third uses
\(\Phi_E\circ\pi_E^*=[\deg\pi_E]\). After identifying
\(\End_\Q(E)\) with \(\Z\), the degree pairing equals
\(a_{r,s}\deg\pi_E\).
The entries \(a_{r,s}\) are integers, and the old degree matrix is symmetric by
the push--pull adjointness of degeneracy maps for the degree pairing
\cite{DerickxOrlic}. It is positive definite by Proposition
\ref{prop:degree-pairing-basic}, because \(\deg\pi_E\cdot A_{E,N}\) is the Gram
matrix of the degree pairing on the \(\Q\)-basis
\(\{\Phi_{r,\mathrm{id}_E}:r\mid N/M\}\) from Theorem
\ref{thm:rational-old-basis}.
The subscript in \(A_{E,N}\) specifies the
elliptic quotient and the level \(N\), while \(M\) is the conductor of \(E\). In
the squarefree or coprime range, the explicit formula of
\cite{DerickxOrlic} computes the entries \(a_{r,s}\).

The matrix \(A_{E,N}\) relates the coordinates of an associated homomorphism
in the rational old basis to the degree of the corresponding morphism.

\begin{proposition}\label{prop:degree}
With the notation of Theorem \ref{thm:rational-old-basis}, let
\(A_{E,N}\) be the old degree matrix defined by
\eqref{eq:old-degree-matrix-entry}. Let
\(g:X_0(N)\to E'\) be a nonconstant morphism over \(\Q\). Let
\(G\in\Hom_\Q(J_0(N),E')\) be its associated homomorphism based at the cusp
\(\infty\), as in Lemma \ref{lem:curve-jacobian}.
The homomorphism \(G\) has an expansion
\begin{equation}\label{eq:degree-old-basis-expansion}
        G=\nu^{-1}\sum_{r\mid N/M}b_r\Phi_{r,u}
        \quad\text{in }\Hom_\Q(J_0(N),E')\otimes_\Z\Q,
        \qquad b_r\in\Z,\quad \nu\in\Z_{>0}.
\end{equation}
For \(\boldsymbol{b}=(b_r)_{r\mid N/M}\), one has
\begin{equation}\label{eq:degree-formula-old}
 \deg g=\deg\pi_E\cdot\deg u\cdot
 \frac{\boldsymbol{b}^{T} A_{E,N}\boldsymbol{b}}{\nu^2}.
\end{equation}
\end{proposition}

\begin{proof}
The existence of \eqref{eq:degree-old-basis-expansion} follows from Theorem
\ref{thm:rational-old-basis}.
Since the degeneracy maps send the cusp \(\infty\) to \(\infty\), the curve map
corresponding to \(\Phi_{r,u}\) is
\(\Phi_{r,u}\circ j_N=u\circ\pi_E\circ\iota_r\).
By Definition \ref{def:degree-pairing}, the push--pull composite for the two maps
\(\pi_E\circ\iota_r\) and \(\pi_E\circ\iota_s\) is given by
\eqref{eq:pairing-on-old-translates}.
Composing both maps with \(u:E\to E'\) multiplies the pairing by \(\deg u\),
since \(u_*u^*=[\deg u]\). Hence
\begin{equation}\label{eq:pairing-on-u-old-translates}
        \langle \Phi_{r,u}\circ j_N,\Phi_{s,u}\circ j_N\rangle
        =a_{r,s}\deg\pi_E\deg u.
\end{equation}
The homomorphism \(G\) associated with \(g\) satisfies
\(G\circ j_N=t_{-g(\infty)}\circ g\). Translation on \(E'\) does not change
the degree. Composing \eqref{eq:degree-old-basis-expansion} with \(j_N\) yields
\(G\circ j_N=\nu^{-1}\sum_{r\mid N/M}b_r(\Phi_{r,u}\circ j_N)\) in the rational
vector space generated by the old pointed maps \(X_0(N)\to E'\). The degree
pairing extends bilinearly to the rational vector space generated by the old
pointed maps. Using bilinearity and the diagonal identity
\eqref{eq:degree-pairing-diagonal}, we obtain
\[
        \deg g
        =\deg(G\circ j_N)
        =\langle G\circ j_N,G\circ j_N\rangle
        =\frac1{\nu^2}\sum_{r,s\mid N/M}b_rb_s
          \langle \Phi_{r,u}\circ j_N,\Phi_{s,u}\circ j_N\rangle .
\]
Substitution of \eqref{eq:pairing-on-u-old-translates} yields
\[
        \deg g
        =\deg\pi_E\cdot\deg u\cdot
          \frac{\sum_{r,s\mid N/M}b_ra_{r,s}b_s}{\nu^2}
        =\deg\pi_E\cdot\deg u\cdot
          \frac{\boldsymbol{b}^{T} A_{E,N}\boldsymbol{b}}{\nu^2}.
\]
\end{proof}

\section{Denominator estimates and the main theorems}
\label{sec:denominators}

Throughout this section, \(E/\Q\) is the strong Weil curve of conductor \(M\),
\(E'/\Q\) is an elliptic curve in the \(\Q\)-isogeny class of \(E\), and
\(u:E\to E'\) is a generator of \(\Hom_\Q(E,E')\). Put
\((u\circ\pi_E)^*\omega_{E'}=c_{u\circ\pi_E}\eta\), with
\(\eta=f_E(q)dq/q\), where \(\omega_{E'}\) is a minimal N\'eron differential
and \(c_{u\circ\pi_E}>0\). Let
\(N\) be a multiple of \(M\), and put \(R=N/M\). For \(r\mid R\), write
\(\iota_r=\iota_{r,N,M}\) and set
\(\Phi_{r,u,N,M}=u\circ\Phi_E\circ\iota_{r,*}:J_0(N)\to E'\).
Also put \(\eta_r=\iota_r^{*}\eta\).
When \(N\), \(M\), and \(u\) are fixed, we write \(\Phi_{r,u}\) for
\(\Phi_{r,u,N,M}\). We prove denominator bounds for the coordinates of elements
of \(\Hom_\Q(J_0(N),E')\) in the \(\Q\)-basis
\(\{\Phi_{r,u}:r\mid R\}\). Define
\begin{equation}\label{eq:D-factor}
        D_{u\circ\pi_E}(R)=
        \begin{cases}
        1, & R=1,\\[3pt]
        \displaystyle
        \prod_{\ell}\ell^{\ord_\ell(c_{u\circ\pi_E})\max\{1,\ord_\ell(R)\}},
        & R>1.
        \end{cases}
\end{equation}

\begin{lemma}\label{lem:first-n}
Let \(G\in\Hom_\Q(J_0(N),E')\). Suppose that the equality
\[
        G=\frac{1}{\nu}\sum_{r\mid R}b_r\Phi_{r,u},
        \qquad b_r\in\Z,\quad \nu\in\Z_{>0}
\]
holds in \(\Hom_\Q(J_0(N),E')\otimes_\Z\Q\).
Put \(\boldsymbol{b}=(b_r)_{r\mid R}\). Then every coordinate of
\(A_{E,N}\boldsymbol{b}\) is divisible by \(\nu\). In particular,
\(\nu\mid \boldsymbol{b}^{T} A_{E,N}\boldsymbol{b}\).
\end{lemma}

\begin{proof}
Fix \(s\mid R\). Since \(G\in\Hom_\Q(J_0(N),E')\), the composite
\(G\circ\iota_s^{*}:J_0(M)\to E'\) belongs to
\(\Hom_\Q(J_0(M),E')\). After tensoring with \(\Q\), the old-basis expansion is
\[
        G\circ\iota_s^{*}
        =
        \frac1\nu\sum_{r\mid R}
        b_r\,u\circ\Phi_E\circ\iota_{r,*}\circ\iota_s^{*}.
\]
By the defining relation \eqref{eq:old-degree-matrix-entry}, we obtain
\begin{equation}\label{eq:first-denominator-composite}
        G\circ\iota_s^{*}
        =\frac{(A_{E,N}\boldsymbol{b})_s}{\nu}\,u\circ\Phi_E .
\end{equation}
Here symmetry of \(A_{E,N}\) identifies
\(\sum_{r\mid R}b_ra_{r,s}\) with the \(s\)-th coordinate of
\(A_{E,N}\boldsymbol{b}\).

Lemma \ref{lem:optimal-primitivity} states that \(u\circ\Phi_E\) is primitive in
\(\Hom_\Q(J_0(M),E')\). Since the left-hand side of
\eqref{eq:first-denominator-composite} belongs to
\(\Hom_\Q(J_0(M),E')\), the
scalar \((A_{E,N}\boldsymbol{b})_s/\nu\) is an integer. This holds for every
\(s\mid R\), so every coordinate of \(A_{E,N}\boldsymbol{b}\) is divisible by
\(\nu\). Multiplication by the integral row vector \(\boldsymbol{b}^{T}\)
proves the final divisibility.
\end{proof}

\subsection{The two-cusp estimate}

The first estimate uses only the cusp \(\infty\) and its Fricke translate. It
does not require a description of the full lattice of differentials integral at
all cusps.

\begin{lemma}\label{lem:q-integral}
Let \(N\) be a positive integer, and let \(A/\Q\) be an abelian variety with
N\'eron model \(\mathcal A/\Z\). Let
\(\omega\in H^0(\mathcal A,\Omega^1_{\mathcal A/\Z})\) be an invariant
differential, and let \(h:X_0(N)\to A\) be a morphism over \(\Q\). Then the
\(q\)-expansion at \(\infty\) of \(h^*\omega\) is integral. The same conclusion
holds for \((h\circ w_N)^*\omega\), where \(w_N\) is the Fricke involution.
\end{lemma}

\begin{proof}
Let \(\mathcal X_0(N)^\infty\) be the \(\Z\)-smooth open subscheme obtained by
removing the fibral irreducible components that do not meet the cusp \(\infty\),
as in \cite{CesnaviciusNeururerSaha}. Its generic fiber contains
\(X_0(N)\) near \(\infty\), and its completion along the cusp is
\(\operatorname{Spf}(\Z[[q]])\). By the N\'eron mapping property
\cite{BLR}, the restriction of \(h\) to the generic
fiber of this smooth open extends uniquely to a morphism
\(\mathcal X_0(N)^\infty\to\mathcal A\). Therefore \(h^*\omega\) is a regular
relative differential on \(\mathcal X_0(N)^\infty\). On the formal completion it
has the form \(a(q)dq\) with \(a(q)\in\Z[[q]]\), equivalently
\(h^*\omega=qa(q)dq/q\), with integral Fourier coefficients. The N\'eron
mapping property applied to the \(\Q\)-morphism \(h\circ w_N\)
proves the second assertion.
\end{proof}

We apply Lemma \ref{lem:q-integral} to the two Fricke-related cusps
to obtain the following coefficient bound.

\begin{lemma}\label{lem:two-cusp}
Let \((x_r)_{r\mid R}\) be a family of rational numbers, and suppose that
\(\omega=\sum_{r\mid R}x_r\eta_r\) is the pullback of a N\'eron differential
under a morphism from \(X_0(N)\) to an abelian variety over \(\Q\).
Then
\begin{equation}\label{eq:gcd-condition}
        \gcd(r,R/r)x_r\in\Z
        \qquad(r\mid R).
\end{equation}
\end{lemma}

\begin{proof}
At \(\infty\), Lemma \ref{lem:q-expansion-degeneracy} yields
\(\eta_r=r\sum_{n\ge1}a_nq^{rn}dq/q\).
The coefficients of \(\omega\) are integral by Lemma \ref{lem:q-integral}. Order
the divisors of \(R\) increasingly. In the coefficient of \(q^r\), the term
indexed by \(r\) contributes \(rx_r\) because \(a_1=1\), while all other
contributions come from proper divisors \(r'\mid r\), \(r'<r\). By induction,
\(r'x_{r'}\in\Z\), and the Fourier coefficients \(a_n\) of \(f_E\) are
integers. Hence every earlier contribution to the coefficient of \(q^r\) is
integral.
Thus
\begin{equation}\label{eq:dx}
        rx_r\in\Z .
\end{equation}

Let \(w_N(z)=-1/(Nz)\). By the Atkin--Lehner eigenvalue relation for newforms
\cite{AtkinLehner,Li}, choose \(\lambda_E\in\{\pm1\}\) so that
\(f_E(-1/(Mz))=\lambda_E Mz^2f_E(z)\).
The transformation formula is
\begin{equation}\label{eq:fricke}
        w_N^*\eta_r=\lambda_E\eta_{R/r}.
\end{equation}
The equality follows from
\[
        \frac{r}{Nz^2}f_E\left(-\frac{r}{Nz}\right)dz
        =\lambda_E\frac{R}{r}f_E\left(\frac{R}{r}z\right)dz.
\]
The differential \(w_N^*\omega\) is the pullback of the chosen N\'eron
differential under the composite with \(w_N\). By \eqref{eq:fricke},
\(w_N^*\omega=\lambda_E\sum_{e\mid R}x_{R/e}\eta_e\). Applying the
triangular argument to the expansion of \(w_N^*\omega\) shows that
\(ex_{R/e}\in\Z\) for every \(e\mid R\). Taking \(e=R/r\) yields
\begin{equation}\label{eq:rdx}
        (R/r)x_r\in\Z .
\end{equation}
A B\'ezout combination of \eqref{eq:dx} and \eqref{eq:rdx} proves
\eqref{eq:gcd-condition}.
\end{proof}

At a fixed prime \(\ell\), the two-cusp estimate already controls the
coefficients indexed by divisors prime to \(\ell\).

\begin{proposition}\label{prop:remove}
Let \(\ell\) be a prime, write \(a=\ord_\ell(c_{u\circ\pi_E})\), and let \(\Z_{(\ell)}\) denote
the localization of \(\Z\) at \(\ell\). Suppose that an element
\(G\in\Hom_\Q(J_0(N),E')\otimes_\Z\Z_{(\ell)}\) has old-basis expansion
\[
        G=\sum_{r\mid R}x_r\Phi_{r,u}.
\]
Then \(\ell^a x_r\in\Z_{(\ell)}\) for every \(r\mid R\) with \(\ell\nmid r\).
\end{proposition}

\begin{proof}
Choose an integer \(n\) prime to \(\ell\) such that
\(nG\in\Hom_\Q(J_0(N),E')\). The pointed curve map \((nG)\circ j_N\) pulls
\(\omega_{E'}\) back to
\[
        n\sum_{r\mid R}x_r(\Phi_{r,u}\circ j_N)^*\omega_{E'}
        =
        nc_{u\circ\pi_E}\sum_{r\mid R}x_r\eta_r .
\]
Lemma \ref{lem:two-cusp} implies
\(\gcd(r,R/r)nc_{u\circ\pi_E}x_r\in\Z\). If \(\ell\nmid r\), then
\(\ell\nmid\gcd(r,R/r)\). Since \(n\) and \(c_{u\circ\pi_E}/\ell^a\) are units in
\(\Z_{(\ell)}\), one has \(\ell^ax_r\in\Z_{(\ell)}\).
\end{proof}

\subsection{Repeated-prime descent}

When \(\ell^2\mid R\), Proposition \ref{prop:remove} does not by itself control
the coefficients indexed by divisors divisible by \(\ell\). We descend from
level \(N\) to level \(N/\ell\) through the corresponding degeneracy map.
Total ramification at \(\infty\) implies that its pushforward on Jacobians is
surjective with geometrically connected kernel. Precomposition with this
pushforward then has saturated image on Hom groups, so integrality descends
from level \(N\) to level \(N/\ell\).

\begin{proposition}\label{prop:degeneracy-degree-ramification}
Let \(N\) be a positive integer, and let \(\ell\) be a prime with
\(\ell^2\mid N\). Put \(N'=N/\ell\), and let
\(\delta_\ell=\iota_{\ell,N,N'}:X_0(N)\to X_0(N')\) be the degeneracy map
analytically induced by \(\tau\mapsto \ell\tau\).
Then \(\deg(\delta_\ell)=\ell\), and \(\delta_\ell\) is totally ramified at
\(\infty\). More
precisely, if \(q_N\) and \(q_{N'}\) denote the standard parameters at
\(\infty\), then
\begin{equation}\label{eq:q-total-ramification}
        \delta_\ell^*(q_{N'})=q_N^\ell,
\end{equation}
and \(e_\infty(\delta_\ell)=\ell=\deg(\delta_\ell)\).
\end{proposition}

\begin{proof}
We first compute the degree from the corresponding congruence-subgroup index.
We then compare local parameters at \(\infty\).

Let \(\alpha=\left(\begin{smallmatrix}\ell&0\\0&1\end{smallmatrix}\right)\).
Set \(H=\alpha\Gamma_0(N)\alpha^{-1}\). We claim that
\[
        H=\Gamma_0(N')\cap\Gamma^0(\ell),
\]
where \(\Gamma^0(\ell)\) denotes the subgroup whose upper-right entry is
divisible by \(\ell\). For
\(\gamma=\left(\begin{smallmatrix}a&b\\c&d\end{smallmatrix}\right)\in
\Gamma_0(N)\), one has
\[
 \alpha\gamma\alpha^{-1}
 =\begin{pmatrix}a&\ell b\\c/\ell&d\end{pmatrix}.
\]
The entry \(c/\ell\) is divisible by \(N'=N/\ell\) because \(c\) is divisible
by \(N\).
Thus
\(\alpha\Gamma_0(N)\alpha^{-1}\subseteq\Gamma_0(N')\cap\Gamma^0(\ell)\).
Conversely, conjugating an element of
\(\Gamma_0(N')\cap\Gamma^0(\ell)\) by \(\alpha^{-1}\) produces an element of
\(\Gamma_0(N)\), so equality holds. By Definition \ref{def:degeneracy}, this
conjugation identifies \(\delta_\ell\) on the open modular curves with the
morphism induced by \(H\subseteq\Gamma_0(N')\).

We compute the index of \(H\) in \(\Gamma_0(N')\). Put
\(T=\left(\begin{smallmatrix}1&1\\0&1\end{smallmatrix}\right)\). For
\(\gamma=\left(\begin{smallmatrix}a&b\\c&d\end{smallmatrix}\right)\in\Gamma_0(N')\),
we have \(c\equiv0\pmod\ell\) because \(\ell\mid N'\). The identity
\(ad-bc=1\) makes \(d\) a unit modulo \(\ell\). There is therefore
a unique \(j\pmod\ell\) such that \(b+jd\equiv0\pmod\ell\). The upper-right
entry of \(T^j\gamma\) is \(b+jd\), so \(T^j\gamma\in H\). Hence every left
coset of \(H\) is represented by one of \(1,T^{-1},\dots,T^{-(\ell-1)}\).
These representatives are distinct because
\(T^{i-j}\in H\) holds exactly when
\(i\equiv j\pmod\ell\). Consequently, \([\Gamma_0(N'):H]=\ell\). Both groups
contain \(-I\), so the same index is obtained after passage to
\(\mathrm{PSL}_2(\mathbb R)\). Thus the induced map of open modular curves has
degree \(\ell\), as does its extension to the compact curves.

At the cusp \(\infty\) on \(X_0(N)\) and \(X_0(N')\), the width is one because
\(T\) lies in both \(\Gamma_0(N)\) and \(\Gamma_0(N')\). Thus the standard
completed local parameters are \(q_N=e^{2\pi i\tau}\) and
\(q_{N'}=e^{2\pi i\tau}\). Since \(\delta_\ell\) is induced by
\(\tau\mapsto\ell\tau\), we obtain
\(q_{N'}\circ \delta_\ell=e^{2\pi i\ell\tau}=q_N^\ell\), which is
\eqref{eq:q-total-ramification}. For a finite morphism of smooth
curves, the ramification index at a point is the order of the pullback of a
uniformizer. Hence \(e_\infty(\delta_\ell)=\ell\). Since
\(e_\infty(\delta_\ell)=\deg(\delta_\ell)\), the point \(\infty\) is the unique
point above the target cusp \(\infty\), and the ramification there is total.
\end{proof}

The total ramification in Proposition
\ref{prop:degeneracy-degree-ramification} precludes any nontrivial \'etale
intermediate factor.

\begin{proposition}\label{prop:connected-degeneracy-kernel}
With the notation of Proposition
\ref{prop:degeneracy-degree-ramification},
\((\delta_\ell)_*:J_0(N)\to J_0(N')\) is surjective and
\(K_\ell:=\ker((\delta_\ell)_*)\) is geometrically
connected.
\end{proposition}

\begin{proof}
By Proposition \ref{prop:degeneracy-degree-ramification}, the map
\(\delta_\ell\) has degree \(\ell\) and is totally ramified at the cusp
\(\infty\).

We first verify that no nontrivial connected finite \'etale intermediate
factor can occur. Suppose that the base change \(\delta_{\ell,\C}\) of
\(\delta_\ell\) to \(\C\) factored as
\(X_0(N)_{\C}\to Y\to X_0(N')_{\C}\), where \(Y\) is connected and
\(Y\to X_0(N')_{\C}\) is \'etale of degree \(m>1\). The ramification index of
\(\delta_{\ell,\C}\) at \(\infty\) would then be at most
\[
        \deg(X_0(N)_{\C}\to Y)
        =
        \frac{\deg(\delta_\ell)}{m}
        <
        \deg(\delta_\ell),
\]
contrary to total ramification.
The image of
\[
 \pi_1\bigl(X_0(N)(\C)\bigr)\longrightarrow
 \pi_1\bigl(X_0(N')(\C)\bigr)
\]
has finite index because \(\delta_{\ell,\C}\) has nonzero degree. If this image
were proper, the corresponding connected finite unramified cover of
\(X_0(N')(\C)\) would give a nontrivial unramified intermediate factor of
\(\delta_{\ell,\C}\), contrary to total ramification. Hence the map on
fundamental groups, and therefore the induced map on \(H_1(-,\Z)\), is
surjective.

The push--pull identity
\((\delta_\ell)_*\circ\delta_\ell^*=[\ell]\) shows that
\((\delta_\ell)_*\) is surjective. Under the analytic uniformization of the
Jacobians \cite{BirkenhakeLange}, its map on period lattices is
identified with the induced map
\[
 H_1\bigl(X_0(N)(\C),\Z\bigr)\longrightarrow
 H_1\bigl(X_0(N')(\C),\Z\bigr).
\]
If a surjective homomorphism \(V/\Lambda\to W/\Lambda'\) of complex
tori is induced by a linear map \(F:V\to W\) satisfying
\(F(\Lambda)\subseteq\Lambda'\), then the component group of its kernel is
\(\Lambda'/F(\Lambda)\). The surjectivity on integral homology therefore
implies that the kernel is connected over \(\C\). Thus \(K_\ell\) is
geometrically connected over \(\Q\).
\end{proof}

Because \(\ker((\delta_\ell)_*)\) is connected, a homomorphism that
becomes divisible by an integer after precomposition with
\((\delta_\ell)_*\) is already divisible by that integer.

\begin{lemma}\label{lem:connected-descent}
With the notation of Proposition
\ref{prop:connected-degeneracy-kernel}, let
\(A/\Q\) be an abelian variety. Suppose that
\(H_0\circ(\delta_\ell)_*=nG\) for an integer \(n\ge1\), where
\(H_0\in\Hom_\Q(J_0(N'),A)\) and \(G\in\Hom_\Q(J_0(N),A)\). Then there is a
unique \(H_1\in\Hom_\Q(J_0(N'),A)\) such that \(H_0=nH_1\) and
\(G=H_1\circ(\delta_\ell)_*\).
\end{lemma}

\begin{proof}
Proposition \ref{prop:connected-degeneracy-kernel} shows that
\((\delta_\ell)_*\) is surjective with geometrically connected kernel.
Let \(K_\ell=\ker((\delta_\ell)_*)\). The equality
\(H_0\circ(\delta_\ell)_*=nG\) implies that the image of \(G|_{K_\ell}\) is
contained in the finite \'etale group scheme \(A[n]\). Since \(K_\ell\) is
geometrically connected,
the restriction \(G|_{K_\ell}\) is constant. Because it is a homomorphism,
it is zero.
By the universal property of the quotient
\(J_0(N)/K_\ell\simeq J_0(N')\)
\cite{Poonen}, \(G\) factors uniquely as
\(G=H_1\circ(\delta_\ell)_*\) for a homomorphism \(H_1:J_0(N')\to A\). Substitution and
the surjectivity of \((\delta_\ell)_*\) imply \(H_0=nH_1\).
\end{proof}

For a fixed prime \(p\), clearing denominators prime to \(p\) and
applying Lemma \ref{lem:connected-descent} proves that precomposition with
\((\delta_\ell)_*\) has saturated image over \(\Z_{(p)}\).

\begin{corollary}\label{cor:connected-local}
With the notation of Proposition \ref{prop:connected-degeneracy-kernel}, let
\(A/\Q\) be an abelian variety. Precomposition with \((\delta_\ell)_*\) has
saturated image after localization at any prime \(p\). More precisely, let
\(H\in\Hom_\Q(J_0(N'),A)\otimes_\Z\Q\). If
\(H\circ(\delta_\ell)_*\in
\Hom_\Q(J_0(N),A)\otimes_\Z\Z_{(p)}\),
then \(H\in\Hom_\Q(J_0(N'),A)\otimes_\Z\Z_{(p)}\).
\end{corollary}

\begin{proof}
Choose an integer \(t\) prime to \(p\) such that
\(t(H\circ(\delta_\ell)_*)\) is integral, and then choose \(n\geq1\) such that
\(ntH\) is integral. Then
\[
        (ntH)\circ(\delta_\ell)_*
        =n\bigl(t(H\circ(\delta_\ell)_*)\bigr).
\]
By Lemma \ref{lem:connected-descent}, there is
\(H_1\in\Hom_\Q(J_0(N'),A)\) such that \(ntH=nH_1\). Hom groups of abelian
varieties are torsion-free, since the image of a torsion homomorphism would be
both connected and contained in a finite \'etale group scheme. Hence
\(tH=H_1\). The integer \(t\) is a unit in \(\Z_{(p)}\), so
\(H\in\Hom_\Q(J_0(N'),A)\otimes_\Z\Z_{(p)}\).
\end{proof}

\subsection{The denominator theorem}

The two-cusp integrality criterion of Lemma \ref{lem:two-cusp}, the
prime-to-\(\ell\) estimate of Proposition \ref{prop:remove}, and the localized
saturation statement of Corollary \ref{cor:connected-local} together establish
the denominator bound in Theorem \ref{thm:denominator}.

\begin{theorem}\label{thm:denominator}
Let \(D_{u\circ\pi_E}(R)\) be defined by \eqref{eq:D-factor}. If
\[
        G=\sum_{r\mid R}x_r\Phi_{r,u}\in\Hom_\Q(J_0(N),E')
\]
is the expansion of \(G\) in the \(\Q\)-basis of Theorem
\ref{thm:rational-old-basis}, then
\begin{equation}\label{eq:global-denominator}
        D_{u\circ\pi_E}(R)x_r\in\Z
        \qquad(r\mid R).
\end{equation}
Equivalently,
\(D_{u\circ\pi_E}(R)\Hom_\Q(J_0(N),E')\subseteq
\bigoplus_{r\mid R}\Z\Phi_{r,u}\).
\end{theorem}

\begin{proof}
The case \(R=1\) is Lemma \ref{lem:optimal-primitivity} applied with
\(E'\) and \(u\), because then
\(\Phi_{1,u}=u\circ\Phi_E\) generates \(\Hom_\Q(J_0(M),E')\). Hence
\(x_1\in\Z\).
Assume \(R>1\) and fix a prime \(\ell\). Put
\(a=\ord_\ell(c_{u\circ\pi_E})\) and \(m=\ord_\ell(R)\). We prove
\begin{equation}\label{eq:local-denominator}
        \ell^{a\max\{1,m\}}x_r\in\Z_{(\ell)}
        \qquad(r\mid R)
\end{equation}
by induction on \(m\).

If \(m\le1\), Lemma \ref{lem:two-cusp}, applied to the pullback
\[
        (G\circ j_N)^*\omega_{E'}
        =
        c_{u\circ\pi_E}\sum_{r\mid R}x_r\eta_r ,
\]
implies \(\gcd(r,R/r)c_{u\circ\pi_E}x_r\in\Z\) for every \(r\mid R\). When
\(m\le1\), both \(\gcd(r,R/r)\) and
\(c_{u\circ\pi_E}/\ell^a\) are units in \(\Z_{(\ell)}\). It follows that
\(\ell^ax_r\in\Z_{(\ell)}\), which is \eqref{eq:local-denominator} in this
case.

Assume that \(m\ge2\). Proposition \ref{prop:remove} implies
\(\ell^ax_r\in\Z_{(\ell)}\) whenever \(\ell\nmid r\).
Subtracting the terms with \(\ell\nmid r\) from \(\ell^aG\)
leaves only those indexed by multiples of \(\ell\), which factor through
\(J_0(N/\ell)\). Set
\[
        G_1=
        \ell^aG-
        \sum_{\substack{r\mid R\\ \ell\nmid r}}\ell^ax_r\Phi_{r,u,N,M}
        \in\Hom_\Q(J_0(N),E')\otimes_\Z\Z_{(\ell)}.
\]
Put \(N'=N/\ell\), \(R'=R/\ell\),
\(\delta_\ell=\iota_{\ell,N,N'}\).
For each \(e\mid R'\), the degeneracy maps compose to give
\[
        \Phi_{\ell e,u,N,M}
        =
        \Phi_{e,u,N',M}\circ(\delta_\ell)_*.
\]
Equivalently, for each \(e\mid R'\), the diagram
\[
\begin{tikzcd}[column sep=large,row sep=large]
J_0(N) \arrow[r,"(\delta_\ell)_*"] \arrow[dr,swap,"\Phi_{\ell e,u,N,M}"] &
J_0(N') \arrow[d,"\Phi_{e,u,N',M}"] \\
& E'
\end{tikzcd}
\]
commutes.
Collecting the terms with index \(\ell e\), we obtain
\[
        G_1=H\circ(\delta_\ell)_*,
        \qquad
        H=\sum_{e\mid R'}\ell^ax_{\ell e}\Phi_{e,u,N',M}
        \quad\text{in }\Hom_\Q(J_0(N'),E')\otimes_\Z\Q.
\]
Corollary \ref{cor:connected-local} shows that
\(H\in\Hom_\Q(J_0(N'),E')\otimes_\Z\Z_{(\ell)}\).
Choose an integer \(t\) prime to \(\ell\) such that
\(tH\in\Hom_\Q(J_0(N'),E')\). Apply the induction hypothesis at level \(N'\)
to \(tH\). Since \(t\) is a unit in \(\Z_{(\ell)}\) and
\(\ord_\ell(R')=m-1\), we have
\[
        \ell^{a(m-1)}\bigl(\ell^ax_{\ell e}\bigr)
        =\ell^{am}x_{\ell e}\in\Z_{(\ell)}.
\]
For \(\ell\nmid r\), Proposition \ref{prop:remove} already proves
\(\ell^ax_r\in\Z_{(\ell)}\), which is stronger than the required bound because
\(m\ge2\). This completes the induction and proves
\eqref{eq:local-denominator}. Combining these local statements over all
primes \(\ell\) proves \eqref{eq:global-denominator}.
\end{proof}

If \(c_{u\circ\pi_E}=1\), then \(D_{u\circ\pi_E}(R)=1\), and
\eqref{eq:global-denominator} implies that every homomorphism has integral
coordinates in the rational old basis.

\begin{corollary}\label{cor:saturation}
If \(c_{u\circ\pi_E}=1\), then the old maps \(\Phi_{r,u}\), for \(r\mid R\),
form a \(\Z\)-basis of \(\Hom_\Q(J_0(N),E')\).
\end{corollary}

\begin{proof}
They form a \(\Q\)-basis by Theorem \ref{thm:rational-old-basis}. When
\(c_{u\circ\pi_E}=1\),
Theorem \ref{thm:denominator} implies that every element of
\(\Hom_\Q(J_0(N),E')\) has integral coefficients in this basis.
\end{proof}

For every \(r\mid R\), one has
\(\Phi_{r,u}=u\circ\Phi_{r,\mathrm{id}_E}\). Hence the integral expansion in
Corollary \ref{cor:saturation} factors every homomorphism through \(u\).

\begin{corollary}\label{cor:factor-through-u}
Assume that \(c_{u\circ\pi_E}=1\). Then every
\(H\in\Hom_\Q(J_0(N),E')\) factors uniquely through \(u:E\to E'\). More precisely,
there is a unique
\(H_E\in \bigoplus_{r\mid R}\Z\,\Phi_{r,\mathrm{id}_E}\) such that
\(H=u\circ H_E\). Consequently, for every morphism
\(g:X_0(N)\to E'\) over \(\Q\), the pointed translate
\(t_{-g(\infty)}\circ g\) factors as \(u\circ h\) for a morphism
\(h:X_0(N)\to E\) over \(\Q\).
\end{corollary}

\begin{proof}
By Corollary \ref{cor:saturation}, write
\(H=\sum_{r\mid R}n_r\Phi_{r,u}\) with \(n_r\in\Z\). Since
\(\Phi_{r,u}=u\circ\Phi_{r,\mathrm{id}_E}\), the homomorphism
\(H_E=\sum_{r\mid R}n_r\Phi_{r,\mathrm{id}_E}\) satisfies \(H=u\circ H_E\). If
\(u\circ H_E=u\circ H'_E\), then
\([\deg u]\circ(H_E-H'_E)=0\) after composing with the dual isogeny
\(\widehat u:E'\to E\). The torsion-freeness of
\(\Hom_\Q(J_0(N),E)\) implies \(H_E=H'_E\). The assertion for pointed morphisms
follows from the universal
property of the Jacobian \eqref{eq:jacobian-universal-property}.
\end{proof}

\begin{remark}
The hypothesis \(c_{u\circ\pi_E}=1\) in Corollary
\ref{cor:factor-through-u} cannot in general be omitted. Gonz\'alez identifies
the elliptic quotient \(X_0(95)/\langle w_{95}\rangle\) with the curve denoted
\(19\mathrm{A}3\) in Cremona's tables, namely the curve \(E'\) with LMFDB label
\(\href{https://www.lmfdb.org/EllipticCurve/Q/19/a/3}{19.\mathrm{a}3}\), where
\(w_{95}\) is the Fricke involution
\cite[Table~2]{GonzalezBielliptic}. Hence the quotient map
\(g:X_0(95)\to E'\) has degree \(2\). On the other hand, the strong Weil curve
\(E\) in this isogeny class has LMFDB label
\(\href{https://www.lmfdb.org/EllipticCurve/Q/19/a/2}{19.\mathrm{a}2}\), and
Cremona's tables contain a degree-\(3\) isogeny \(u:E\to E'\)
\cite[Chapter~IV, Table~1]{CremonaAlgorithms}. Since \(3\nmid2\), no translate
of \(g\) can factor as \(u\circ h\) with \(h:X_0(95)\to E\). In particular,
\(c_{u\circ\pi_E}\ne1\) in this example.
\end{remark}

\subsection{Consequences and the main divisibility theorem}

We first establish the integral-basis theorem.

\begin{proof}[Proof of Theorem \ref{thm:lattice-main}]
The assertion follows from Corollary \ref{cor:saturation}.
\end{proof}

We next turn to Theorem \ref{thm:main}, beginning with a
fixed-target divisibility statement.

\begin{proposition}\label{prop:fixed-target-degree}
Every nonconstant morphism \(g:X_0(N)\to E'\) over \(\Q\) satisfies
\[
        \deg(u\circ\pi_E)\mid D_{u\circ\pi_E}(R)\deg g .
\]
\end{proposition}

\begin{proof}
Translate \(g\) by \(-g(\infty)\). Translation does not change the degree, so we
may assume that \(g(\infty)=0\). By Lemma \ref{lem:curve-jacobian}, the map \(g\)
corresponds to a homomorphism \(G:J_0(N)\to E'\). By Theorem
\ref{thm:rational-old-basis}, write
\(G=\nu^{-1}\sum_{r\mid R}b_r\Phi_{r,u}\) with \(b_r\in\Z\),
\(\nu\in\Z_{>0}\), and
\(\gcd(\nu,\{b_r:r\mid R\})=1\).
Put \(\boldsymbol b=(b_r)_{r\mid R}\).
By Theorem \ref{thm:denominator}, applied to
the coordinates \(x_r=b_r/\nu\), one has
\(D_{u\circ\pi_E}(R)b_r/\nu\in\Z\) for every \(r\mid R\). The coprimality
condition forces
\begin{equation}\label{eq:nu-divides-Df}
        \nu\mid D_{u\circ\pi_E}(R).
\end{equation}
Moreover, by Lemma \ref{lem:first-n},
\begin{equation}\label{eq:nu-divides-quadratic}
        \nu\mid\boldsymbol{b}^{T}A_{E,N}\boldsymbol{b}.
\end{equation}
By Lemma \ref{lem:fixed-target-manin}, \(\deg(u\circ\pi_E)=\deg\pi_E\deg u\).
The degree formula \eqref{eq:degree-formula-old} is therefore
\[
        \frac{D_{u\circ\pi_E}(R)\deg g}{\deg(u\circ\pi_E)}
        =
        D_{u\circ\pi_E}(R)
        \frac{\boldsymbol{b}^{T}A_{E,N}\boldsymbol{b}}{\nu^2}.
\]
The right-hand side is an integer by \eqref{eq:nu-divides-Df} and
\eqref{eq:nu-divides-quadratic}. Hence \(\deg(u\circ\pi_E)\mid D_{u\circ\pi_E}(R)\deg g\).
\end{proof}

When \(c_{u\circ\pi_E}=1\), the integral old basis sharpens
Proposition \ref{prop:fixed-target-degree} to an exact description of the
degrees of pointed morphisms.

\begin{corollary}\label{cor:fixed-target-degree-form}
With the notation of Proposition \ref{prop:fixed-target-degree}, let
\(A_{E,N}\) be the old degree matrix defined by
\eqref{eq:old-degree-matrix-entry}. Suppose that \(c_{u\circ\pi_E}=1\).
For every nonconstant morphism \(g:X_0(N)\to E'\), its associated
homomorphism \(G_g\) based at \(\infty\) has a unique nonzero coefficient
vector \(\boldsymbol b=(b_r)_{r\mid R}\in\Z^{\{r:\,r\mid R\}}\) satisfying
\begin{equation}\label{eq:fixed-target-exact-degree}
        G_g=\sum_{r\mid R}b_r\Phi_{r,u},
        \qquad
        \deg g=\deg(u\circ\pi_E)\,
        \boldsymbol b^{T}A_{E,N}\boldsymbol b.
\end{equation}
Conversely, every nonzero
\(\boldsymbol b\in\Z^{\{r:\,r\mid R\}}\) defines a nonconstant pointed morphism
\[
        g_{\boldsymbol b}
        =
        \left(\sum_{r\mid R}b_r\Phi_{r,u}\right)\circ j_N,
\]
whose degree is given by \eqref{eq:fixed-target-exact-degree} with
\(g=g_{\boldsymbol b}\).
\end{corollary}

\begin{proof}
Corollary \ref{cor:saturation} provides the unique integral expansion of \(G_g\),
and \eqref{eq:degree-formula-old} computes its degree. Conversely, a
nonzero vector \(\boldsymbol b\) defines a nonzero homomorphism because the old
maps form a \(\Z\)-basis. Since the image of the Abel--Jacobi map
\(j_N\)
generates \(J_0(N)\), its composite with this homomorphism is nonconstant.
Equation \eqref{eq:degree-formula-old} computes its degree.
\end{proof}

To obtain the target-independent factor in Theorem \ref{thm:main}, we
compose with the dual isogeny \(E'\to E\) and apply the denominator theorem
with target \(E\).

\begin{proof}[Proof of Theorem \ref{thm:main}]
Let \(R=N/M\). Translate \(g\) by \(-g(\infty)\), and let
\(G:J_0(N)\to E'\) be the associated homomorphism. Choose a generator
\(u:E\to E'\) of \(\Hom_\Q(E,E')\), and let \(\widehat u:E'\to E\) be the dual
isogeny. By Theorem \ref{thm:rational-old-basis}, write
\(G=\nu^{-1}\sum_{r\mid R}b_r\Phi_{r,u}\) with \(b_r\in\Z\),
\(\nu\in\Z_{>0}\), and
\(\gcd(\nu,\{b_r:r\mid R\})=1\).
Put \(\boldsymbol b=(b_r)_{r\mid R}\).
Since \(\widehat u\circ u=[\deg u]\), for each
\(r\mid R\) one has
\(\widehat u\circ\Phi_{r,u}=[\deg u]\circ\Phi_{r,\mathrm{id}_E}\). Therefore
\(\widehat u\circ G=(\deg u/\nu)\sum_{r\mid R}b_r\Phi_{r,\mathrm{id}_E}\) in
\(\Hom_\Q(J_0(N),E)\otimes_\Z\Q\).

Apply Theorem \ref{thm:denominator} to the \(E\)-valued homomorphism
\(\widehat u\circ G\), using the target \(E\) and the fixed isogeny
\(\mathrm{id}_E:E\to E\).
Let \(D_E(R)\) denote the integer \eqref{eq:D-factor} in that case, so
\(D_E(1)=1\) and, for \(R>1\),
\[
        D_E(R)=
        \prod_{\ell}\ell^{\ord_\ell(c_E)\max\{1,\ord_\ell(R)\}}.
\]
The denominator theorem places \(D_E(R)(\deg u/\nu)b_r\) in \(\Z\) for every
\(r\mid R\). Since \(\nu\) is coprime to the \(b_r\)'s, these integrality
relations force
\begin{equation}\label{eq:nu-divides-DE-degu}
        \nu\mid D_E(R)\deg u.
\end{equation}
Moreover, by Lemma \ref{lem:first-n},
\begin{equation}\label{eq:nu-divides-quadratic-main}
        \nu\mid\boldsymbol{b}^{T}A_{E,N}\boldsymbol{b}.
\end{equation}
By the degree formula \eqref{eq:degree-formula-old},
\[
        \frac{D_E(R)\deg g}{\deg\pi_E}
        =
        \left(\frac{D_E(R)\deg u}{\nu}\right)
        \left(\frac{\boldsymbol{b}^{T}A_{E,N}\boldsymbol{b}}{\nu}\right).
\]
The right-hand side is an integer by \eqref{eq:nu-divides-DE-degu} and
\eqref{eq:nu-divides-quadratic-main}. Hence
\(\deg\pi_E\mid D_E(R)\deg g\).

For \(R=1\), \(D_E(R)=1=c_E^{\Omega(R)}\). For \(R>1\), equation
\eqref{eq:D-factor} shows that \(D_E(R)\mid c_E^{\Omega(R)}\). Hence
\(\deg\pi_E\mid c_E^{\Omega(R)}\deg g\). Its equivalent form is the elementary
identity \(a\mid cb\) if and only if
\(a/\gcd(a,c)\mid b\).
\end{proof}

\subsection{Mordell--Weil applications of the divisibility theorem}
\label{subsec:rank-consequences}

We apply the divisibility theorem, Theorem~\ref{thm:main}, to bound the
Mordell--Weil ranks of elliptic quotients of modular curves. Let \(E/\Q\) be
the strong Weil curve of conductor
\(M\), let \(\pi_E:X_0(M)\to E\) be its strong Weil parametrization, and let
\(c_E\) be its Manin constant. Let \(N\) be a positive multiple of \(M\), let
\(E'/\Q\) be \(\Q\)-isogenous to \(E\), and let
\(g:X_0(N)\to E'\) be a nonconstant morphism of degree \(d\). Write \(\ord_2\)
for the \(2\)-adic valuation. Theorem~\ref{thm:main} gives
\begin{equation}\label{eq:rank-input-divisibility}
                  \deg\pi_E\mid c_E^{\Omega(N/M)}d.
\end{equation}

Watkins's conjecture supplies the complementary divisibility.

\begin{conjecture}[Watkins \cite{Watkins}]\label{conj:watkins}
Every strong Weil curve \(E/\Q\) satisfies
\begin{equation}\label{eq:watkins-divisibility}
        2^{\operatorname{rank}E(\Q)}\mid\deg\pi_E.
\end{equation}
\end{conjecture}

Equations~\eqref{eq:rank-input-divisibility} and
\eqref{eq:watkins-divisibility} imply a rank bound with no parity restriction
on \(d\).

\begin{theorem}\label{thm:watkins-rank-bound}
Assume that the divisibility in Conjecture~\ref{conj:watkins} holds for
\(E\). For every positive multiple \(N\) of \(M\), every elliptic curve
\(E'/\Q\) that is \(\Q\)-isogenous to \(E\), and every nonconstant morphism
\(g:X_0(N)\to E'\) of degree \(d\), one has
\begin{equation}\label{eq:watkins-rank-bound}
 \operatorname{rank}E'(\Q)
 \leq \Omega(N/M)\ord_2(c_E)+\ord_2(d).
\end{equation}
If \(N=M\) or \(c_E\) is odd, the bound simplifies to
\(\operatorname{rank}E'(\Q)\leq\ord_2(d)\).
\end{theorem}

\begin{proof}
Mordell--Weil rank is invariant under \(\Q\)-isogeny. Taking \(2\)-adic
valuations in \eqref{eq:rank-input-divisibility} and
\eqref{eq:watkins-divisibility} gives
\[
 \operatorname{rank}E'(\Q)
 =\operatorname{rank}E(\Q)
 \leq\ord_2(\deg\pi_E)
 \leq\Omega(N/M)\ord_2(c_E)+\ord_2(d).\qedhere
\]
\end{proof}

\begin{corollary}\label{cor:cremona-rank-bound}
Let \(E/\Q\) be a strong Weil curve of conductor \(M<400000\). For every
positive multiple \(N\) of \(M\), every elliptic curve \(E'/\Q\) that is
\(\Q\)-isogenous to \(E\), and every nonconstant morphism
\(g:X_0(N)\to E'\) of degree \(d\), one has
\(\operatorname{rank}E'(\Q)\leq\ord_2(d)\).
More explicitly, \(\operatorname{rank}E'(\Q)=0\) if \(2\nmid d\),
\(\operatorname{rank}E'(\Q)\leq1\) if \(4\nmid d\), and
\(\operatorname{rank}E'(\Q)\leq2\) if \(8\nmid d\).
\end{corollary}

\begin{proof}
Cremona's tables \cite{CremonaData} verify Watkins's conjecture
and \(c_E=1\) for every strong Weil curve of conductor less than
\(400000\). Theorem~\ref{thm:watkins-rank-bound} therefore applies with
\(c_E=1\).
\end{proof}

Ogg's inequality and Corollary~\ref{cor:cremona-rank-bound} imply the following
unconditional result for odd degrees at most \(1645\).

\begin{theorem}\label{thm:ogg-rank-bound}
For every positive odd integer \(d\leq1645\), no modular curve
\(X_0(N)/\Q\) admits a degree-\(d\) morphism over \(\Q\) to an elliptic curve
of positive \(\Q\)-rank.
\end{theorem}

\begin{proof}
Suppose that \(g:X_0(N)\to E'\) has degree \(d\) and that
\(\operatorname{rank}E'(\Q)>0\). Ogg's inequality \cite{OggHyperelliptic},
together with the Hasse bound, gives \(N\leq399734\). Corollary~\ref{cor:cremona-rank-bound}
then gives \(\operatorname{rank}E'(\Q)=0\), a contradiction.
\end{proof}

The case \(d=5\) recovers the result of
Derickx--Hwang--Jeon--Orli\'c \cite{DerickxHwangJeonOrlic} that no modular
curve \(X_0(N)/\Q\) admits a degree-\(5\) morphism over \(\Q\) to an elliptic
curve of positive \(\Q\)-rank.

Theorem~\ref{thm:watkins-rank-bound} is conditional on Watkins's conjecture.
For odd-degree morphisms, we use the following results of
Agashe--Ribet--Stein and Kazalicki--Kohen.

\begin{theorem}[Agashe--Ribet--Stein
\cite{AgasheRibetSteinCongruence}]
\label{thm:congruence-modular-parity}
If \(4\nmid M\), then the \(2\)-adic valuation of the congruence number of
\(E\) equals \(\ord_2(\deg\pi_E)\).
\end{theorem}

\begin{theorem}[Kazalicki--Kohen \cite{KazalickiKohenCorrigendum}]
\label{thm:kazalicki-kohen-rank-zero}
Let \(E/\Q\) be an elliptic curve. If the congruence number of \(E\) is
odd, then \(\operatorname{rank}E(\Q)=0\).
\end{theorem}

Combining Theorems~\ref{thm:congruence-modular-parity}
and~\ref{thm:kazalicki-kohen-rank-zero} with the divisibility theorem
extends the odd-degree rank-zero criterion from level \(M\) to every higher
level \(N\) divisible by \(M\), provided \(4\nmid M\).

\begin{theorem}\label{thm:odd-degree-rank-zero}
Assume that \(4\nmid M\). Let \(N\) be a positive multiple of \(M\), and let
\(E'/\Q\) be \(\Q\)-isogenous to \(E\). If
\(g:X_0(N)\to E'\) is a nonconstant morphism of odd degree, then
\(\operatorname{rank}E'(\Q)=0\).
\end{theorem}

\begin{proof}
Since \(4\nmid M\), we have \(\ord_2(c_E)=0\) by
\v{C}esnavi\v{c}ius \cite[Theorem~1.2]{CesnaviciusManinSemistable}.
Thus \(\deg\pi_E\) is odd by \eqref{eq:rank-input-divisibility}.
By Theorem~\ref{thm:congruence-modular-parity}, the congruence number of \(E\)
is odd. Theorem~\ref{thm:kazalicki-kohen-rank-zero} then gives
\(\operatorname{rank}E(\Q)=0\). Since \(E\) and \(E'\) are \(\Q\)-isogenous,
\(\operatorname{rank}E'(\Q)=0\).
\end{proof}

\section{A tower criterion and compatible intermediate modular curves}\label{sec:tower-intermediate}

Theorem \ref{thm:denominator} follows from the existence of a rational
old basis, a two-cusp bound for its coefficients, and connected-kernel descent
at repeated primes. We formulate these three properties as a criterion for
compatible towers of curves and verify the criterion for
intermediate modular curves between \(X_1\) and \(X_0\).

Throughout this section, whenever \(M\mid N\), write \(R=N/M\). Recall that
\(\Omega(R)=\sum_p\ord_p(R)\), with \(\Omega(1)=0\), and write
\(\Z_{(\ell)}\) for the localization of \(\Z\) at a prime \(\ell\).

For a positive integer \(c\), set \(\mathcal D_c(1)=1\), and for \(R>1\) put
\begin{equation}\label{eq:sharp-D-definition}
        \mathcal D_c(R)
        =
        \prod_{\ell\mid c}
        \ell^{\ord_\ell(c)+\lfloor\ord_\ell(R)/2\rfloor}.
\end{equation}

\subsection{An abstract tower criterion}

\begin{proposition}
\label{prop:connected-quotient-saturation}
Let \(q:B\to C\) be a surjective homomorphism of abelian varieties over \(\Q\)
with geometrically connected kernel, and let \(A/\Q\) be an abelian variety.
Precomposition with \(q\) defines an injective homomorphism with saturated image
\[
        \Hom_\Q(C,A)\longrightarrow\Hom_\Q(B,A),
        \qquad H\longmapsto H\circ q.
\]
More precisely, suppose that \(H\circ q=nG\) for some integer \(n\geq1\), with
\(H\in\Hom_\Q(C,A)\) and \(G\in\Hom_\Q(B,A)\). Then there is a unique
\(H_1\in\Hom_\Q(C,A)\) such that \(H=nH_1\) and \(G=H_1\circ q\).
For every prime \(p\), if
\(H\in\Hom_\Q(C,A)\otimes_\Z\Q\) and
\(H\circ q\in\Hom_\Q(B,A)\otimes_\Z\Z_{(p)}\), then
\(H\in\Hom_\Q(C,A)\otimes_\Z\Z_{(p)}\).
\end{proposition}

\begin{proof}
Injectivity follows from the surjectivity of \(q\). Put \(K=\ker(q)\), and
suppose that \(H\circ q=nG\). The restriction of \(G\) to \(K\) has image in
the finite \'etale group scheme \(A[n]\). Since \(K\) is geometrically
connected,
the restriction \(G|_K\) is constant. Because it is a homomorphism, it is zero.
Since \(G|_K=0\), the homomorphism \(G\) factors uniquely through the
quotient \(B/K\simeq C\) by its universal property
\cite{Poonen}. Thus \(G=H_1\circ q\) for a unique
\(H_1\in\Hom_\Q(C,A)\). Substitution and the
surjectivity of \(q\) imply \(H=nH_1\).

To prove the localized assertion, choose an integer \(t\)
prime to \(p\) such that \(t(H\circ q)\) is integral, and then choose
\(n\geq1\) such that \(ntH\) is integral. The equality
\[
        (ntH)\circ q=n\bigl(t(H\circ q)\bigr)
\]
satisfies the hypotheses of the integral assertion. Therefore there is
\(H_1\in\Hom_\Q(C,A)\) such that \(ntH=nH_1\). Hom groups of abelian varieties
are torsion-free. Indeed, the image of a homomorphism annihilated by \(n\) is
both connected and contained in the finite \'etale group scheme \(A[n]\), and
is therefore zero. It follows that \(tH=H_1\) is integral. Since \(t\) is a
unit in \(\Z_{(p)}\), the localized assertion follows.
\end{proof}

If \(u\circ\Phi\) spans
\(\Hom_\Q(J,E')\otimes_\Z\Q\), Proposition
\ref{prop:connected-quotient-saturation} implies that \(u\circ\Phi\) is
primitive in \(\Hom_\Q(J,E')\).

\begin{corollary}
\label{cor:optimal-quotient-primitivity}
Let \(J/\Q\) be an abelian variety, let \(E/\Q\) be an elliptic curve, and let
\(\Phi:J\to E\) be a surjective homomorphism with geometrically connected
kernel. Let \(E'/\Q\) be isogenous to \(E\), and let \(u:E\to E'\) generate
\(\Hom_\Q(E,E')\). If
\begin{equation}\label{eq:optimal-quotient-rational-span}
        \Hom_\Q(J,E')\otimes_\Z\Q=\Q\,(u\circ\Phi),
\end{equation}
then composition with \(\Phi\) induces an isomorphism
\[
        \Hom_\Q(E,E')\xrightarrow{\sim}\Hom_\Q(J,E'),
        \qquad v\longmapsto v\circ\Phi,
\]
and \(u\circ\Phi\) is primitive in \(\Hom_\Q(J,E')\).
\end{corollary}

\begin{proof}
The cokernel of precomposition with \(\Phi\) is torsion by
\eqref{eq:optimal-quotient-rational-span} and torsion-free by Proposition
\ref{prop:connected-quotient-saturation}. Hence the cokernel is zero, and
precomposition with \(\Phi\) is an isomorphism. Since \(u\) generates
\(\Hom_\Q(E,E')\), its image
\(u\circ\Phi\) is primitive in \(\Hom_\Q(J,E')\).
\end{proof}

Fix a positive integer \(M\). For every positive multiple \(L\) of \(M\), let
\(\mathcal X(L)\) be a smooth projective geometrically connected curve over
\(\Q\), with a chosen point \(P_L\in\mathcal X(L)(\Q)\). Suppose that whenever
\(M\mid L\mid N\) and \(r\mid N/L\), there is a finite morphism
\(\iota^{\mathcal X}_{r,N,L}:\mathcal X(N)\to\mathcal X(L)\) that sends
\(P_N\) to \(P_L\). For every \(s\mid L/M\), require
\begin{equation}\label{eq:tower-map-compatibility}
        \iota^{\mathcal X}_{rs,N,M}
        =
        \iota^{\mathcal X}_{s,L,M}\circ
        \iota^{\mathcal X}_{r,N,L}.
\end{equation}
Equivalently, the following diagram commutes.
\[
\begin{tikzcd}[column sep=large,row sep=large]
\mathcal X(N)
  \arrow[r,"\iota^{\mathcal X}_{r,N,L}"]
  \arrow[dr,"\iota^{\mathcal X}_{rs,N,M}"']
& \mathcal X(L)
  \arrow[d,"\iota^{\mathcal X}_{s,L,M}"] \\
& \mathcal X(M).
\end{tikzcd}
\]
Set \(J_{\mathcal X}(L)=J(\mathcal X(L))\). Let
\(\Phi_M:J_{\mathcal X}(M)\twoheadrightarrow E_0\) be an optimal elliptic
quotient over \(\Q\), so \(\ker(\Phi_M)\) is geometrically connected. Let
\(\pi_M=\Phi_M\circ j_{P_M}:\mathcal X(M)\to E_0\), and let \(E'/\Q\) be in
the \(\Q\)-isogeny class of \(E_0\). Choose a generator
\(u:E_0\to E'\) of \(\Hom_\Q(E_0,E')\), and fix a positive integer
\(c_{\mathcal X,u}\). For every multiple \(N\) of \(M\) and every
\(r\mid N/M\), set
\(\Phi^{\mathcal X}_{r,u,N,M}
=u\circ\Phi_M\circ\iota^{\mathcal X}_{r,N,M,*}\).
We call \(\Phi^{\mathcal X}_{r,u,N,M}\), for \(r\mid N/M\), a level-\(N\)
old map.

Consider the following three conditions for every multiple \(N\) of \(M\), with
\(R=N/M\).
\begin{enumerate}
\item[{\raisebox{\baselineskip}[0pt][0pt]{\hypertarget{tower-condition-T1}{}}\(\mathrm{(T1)}\)}] The level-\(N\) old maps form a \(\Q\)-basis of
\(\Hom_\Q(J_{\mathcal X}(N),E')\otimes_\Z\Q\), and \(u\circ\Phi_M\) is
primitive in \(\Hom_\Q(J_{\mathcal X}(M),E')\).

\item[{\raisebox{\baselineskip}[0pt][0pt]{\hypertarget{tower-condition-T2}{}}\(\mathrm{(T2)}\)}] Let \(\ell\) be any prime, and let \(x_r\in\Q\) for
\(r\mid R\). If
\[
        G=\sum_{r\mid R}x_r\Phi^{\mathcal X}_{r,u,N,M}
        \in\Hom_\Q(J_{\mathcal X}(N),E')\otimes_\Z\Z_{(\ell)},
\]
then
\begin{equation}\label{eq:tower-two-cusp-bound}
        \gcd(r,R/r)c_{\mathcal X,u}x_r\in\Z_{(\ell)}
        \qquad(r\mid R).
\end{equation}

\item[{\raisebox{\baselineskip}[0pt][0pt]{\hypertarget{tower-condition-T3}{}}\(\mathrm{(T3)}\)}] If \(\ell^2\mid R\), put \(N'=N/\ell\). The pushforward
\(\iota^{\mathcal X}_{\ell,N,N',*}:J_{\mathcal X}(N)\to
J_{\mathcal X}(N')\) is surjective with geometrically connected kernel, and
\[
        \Phi^{\mathcal X}_{\ell s,u,N,M}
        =
        \Phi^{\mathcal X}_{s,u,N',M}
        \circ\iota^{\mathcal X}_{\ell,N,N',*}
        \qquad(s\mid N'/M).
\]
\end{enumerate}

\begin{theorem}\label{thm:tower-denominator-criterion}
With the notation of this subsection, assume conditions
\hyperlink{tower-condition-T1}{\(\mathrm{(T1)}\)}--\hyperlink{tower-condition-T3}{\(\mathrm{(T3)}\)}.
Let \(N\) be a multiple of \(M\), and let
\(G=\sum_{r\mid R}x_r\Phi^{\mathcal X}_{r,u,N,M}
\in\Hom_\Q(J_{\mathcal X}(N),E')\).
For a prime \(\ell\) and a divisor \(r\mid R\), put
\(a=\ord_\ell(c_{\mathcal X,u})\), \(m=\ord_\ell(R)\), and
\(k=\ord_\ell(r)\).
Then
\begin{equation}\label{eq:tower-coordinate-bound}
        \begin{cases}
        x_r\in\Z_{(\ell)},&a=0,\\[2mm]
        \ell^{a+\min\{k,m-k\}}x_r\in\Z_{(\ell)},&a>0.
        \end{cases}
\end{equation}
Consequently,
\begin{equation}\label{eq:tower-sharp-lattice}
        \mathcal D_{c_{\mathcal X,u}}(R)
        \Hom_\Q(J_{\mathcal X}(N),E')
        \subseteq
        \bigoplus_{r\mid R}\Z\Phi^{\mathcal X}_{r,u,N,M}.
\end{equation}
In particular, if \(c_{\mathcal X,u}=1\), then the old maps form a
\(\Z\)-basis of \(\Hom_\Q(J_{\mathcal X}(N),E')\).
\end{theorem}

\begin{proof}
If \(R=1\), condition \hyperlink{tower-condition-T1}{\(\mathrm{(T1)}\)} means that \(G=x_1u\circ\Phi_M\). The
primitivity of \(u\circ\Phi_M\) forces \(x_1\in\Z\). This establishes
\eqref{eq:tower-coordinate-bound}, \eqref{eq:tower-sharp-lattice}, and the basis
assertion in this case. Assume \(R>1\), and fix a prime \(\ell\). If \(a>0\), then
\(c_{\mathcal X,u}/\ell^a\) is a unit in
\(\Z_{(\ell)}\). Equation \eqref{eq:tower-two-cusp-bound}, together with
\(\ord_\ell(\gcd(r,R/r))=\min\{k,m-k\}\), yields the second line of
\eqref{eq:tower-coordinate-bound}.

Assume that \(a=0\). We prove \(x_r\in\Z_{(\ell)}\) by induction on \(m\).
If \(m\le1\), then \(\gcd(r,R/r)\) is an \(\ell\)-adic unit for every
\(r\mid R\), so condition \hyperlink{tower-condition-T2}{\(\mathrm{(T2)}\)} implies
\(x_r\in\Z_{(\ell)}\) for every \(r\mid R\). This is the first line of
\eqref{eq:tower-coordinate-bound}. Suppose \(m\ge2\). By condition
\hyperlink{tower-condition-T2}{\(\mathrm{(T2)}\)},
\(x_r\in\Z_{(\ell)}\) whenever \(\ell\nmid r\). After subtracting these terms
from \(G\), write
\[
        \sum_{s\mid R/\ell}x_{\ell s}
        \Phi^{\mathcal X}_{\ell s,u,N,M}
        =
        H\circ\iota^{\mathcal X}_{\ell,N,N/\ell,*},
        \qquad
        H=
        \sum_{s\mid R/\ell}x_{\ell s}
        \Phi^{\mathcal X}_{s,u,N/\ell,M}.
\]
By condition \hyperlink{tower-condition-T3}{\(\mathrm{(T3)}\)}, the pushforward
\(\iota^{\mathcal X}_{\ell,N,N/\ell,*}\) is surjective with geometrically
connected kernel. The localized assertion of Proposition
\ref{prop:connected-quotient-saturation}, applied with
\(q=\iota^{\mathcal X}_{\ell,N,N/\ell,*}\), then shows that \(H\) belongs to
\(\Hom_\Q(J_{\mathcal X}(N/\ell),E')\otimes_\Z\Z_{(\ell)}\).
Choose an integer \(t\) prime to \(\ell\) such that
\(tH\in\Hom_\Q(J_{\mathcal X}(N/\ell),E')\). The induction hypothesis at level
\(N/\ell\), applied to \(tH\), yields
\(t x_{\ell s}\in\Z_{(\ell)}\) for every \(s\mid R/\ell\). Since \(t\) is a
unit in \(\Z_{(\ell)}\), we have \(x_{\ell s}\in\Z_{(\ell)}\). This
completes the proof of the first line of \eqref{eq:tower-coordinate-bound}.

For \(a>0\), the largest possible value of \(\min\{k,m-k\}\), as
\(0\le k\le m\), is \(\lfloor m/2\rfloor\). Combining the local statements
over all primes establishes \eqref{eq:tower-sharp-lattice}. If
\(c_{\mathcal X,u}=1\), then \(\mathcal D_{c_{\mathcal X,u}}(R)=1\). The reverse
inclusion in \eqref{eq:tower-sharp-lattice} holds because every old map is an
integral homomorphism, so the two lattices are equal.
\end{proof}

The degree-pairing matrix relates the old-basis coordinates in
Theorem \ref{thm:tower-denominator-criterion} to the degree of the associated
morphism.

\begin{proposition}\label{prop:tower-degree-matrix}
Assume condition \hyperlink{tower-condition-T1}{\(\mathrm{(T1)}\)}. Let \(N\) be a multiple of \(M\), and write
\(\iota_r=\iota^{\mathcal X}_{r,N,M}\). There is a symmetric positive
definite integral matrix
\(A_{\mathcal X,N}=(a_{r,s})_{r,s\mid R}\) characterized by
\[
        \Phi_M\circ\iota_{r,*}\circ\iota_s^*
        =
        [a_{r,s}]\circ\Phi_M.
\]
Let \(0\ne G\in\Hom_\Q(J_{\mathcal X}(N),E')\), and suppose that
\[
        G=\nu^{-1}\sum_{r\mid R}b_r\Phi^{\mathcal X}_{r,u,N,M},
        \qquad b_r\in\Z,\quad \nu\in\Z_{>0}.
\]
For \(\boldsymbol b=(b_r)_{r\mid R}\), the morphism
\(g=G\circ j_{P_N}:\mathcal X(N)\to E'\) is nonconstant and satisfies
\begin{equation}\label{eq:tower-degree-and-first-denominator}
        \deg g=
        \deg(u\circ\pi_M)
        \frac{\boldsymbol b^{T}A_{\mathcal X,N}\boldsymbol b}{\nu^2},
        \qquad
        \nu\mid\boldsymbol b^{T}A_{\mathcal X,N}\boldsymbol b.
\end{equation}
\end{proposition}

\begin{proof}
At level \(M\), condition \hyperlink{tower-condition-T1}{\(\mathrm{(T1)}\)} and the injectivity of
postcomposition with the isogeny \(u\) imply
\begin{equation}\label{eq:tower-base-rational-span}
        \Hom_\Q(J_{\mathcal X}(M),E_0)\otimes_\Z\Q=\Q\Phi_M.
\end{equation}
In view of \eqref{eq:tower-base-rational-span} and the optimality of
\(\Phi_M\), Corollary \ref{cor:optimal-quotient-primitivity} applies with
target \(E_0\). Hence
\[
        \Hom_\Q(J_{\mathcal X}(M),E_0)=\Z\Phi_M.
\]
The integers \(a_{r,s}\) are therefore well defined.

Put \(H_r=\Phi_M\circ\iota_{r,*}\). Compatibility with the
chosen base points means that
\(H_r\circ j_{P_N}=\pi_M\circ\iota_r\). By the definition
of \(a_{r,s}\) and the push--pull identity,
\begin{align*}
 \langle H_r\circ j_{P_N},H_s\circ j_{P_N}\rangle
 &=\Phi_M\circ\iota_{r,*}\circ\iota_s^*\circ\pi_M^*\\
 &=[a_{r,s}]\circ\Phi_M\circ\pi_M^*
  =[a_{r,s}\deg\pi_M].
\end{align*}
The last equality uses \(\Phi_M\circ\pi_M^*=[\deg\pi_M]\).
Condition \hyperlink{tower-condition-T1}{\(\mathrm{(T1)}\)} says that the maps \(u\circ H_r\) are linearly
independent. Since \(u\) is an isogeny, the maps \(H_r\) are also linearly
independent. By Proposition \ref{prop:degree-pairing-basic},
\((\deg\pi_M)A_{\mathcal X,N}\) is their symmetric positive definite Gram
matrix.

Using the identity \(u_*u^*=[\deg u]\), we have
\begin{equation}\label{eq:tower-old-pairing}
 \langle u\circ H_r\circ j_{P_N},u\circ H_s\circ j_{P_N}\rangle
 =a_{r,s}\deg(u\circ\pi_M).
\end{equation}
Since the image of \(j_{P_N}\) generates \(J_{\mathcal X}(N)\), the assumption
\(G\ne0\) implies that \(g=G\circ j_{P_N}\) is nonconstant. Bilinearity of the
degree pairing and \eqref{eq:tower-old-pairing} yield
\[
 \deg g
 =\frac1{\nu^2}\sum_{r,s\mid R}b_rb_s
   \langle u\circ H_r\circ j_{P_N},u\circ H_s\circ j_{P_N}\rangle
 =\deg(u\circ\pi_M)
  \frac{\boldsymbol b^{T}A_{\mathcal X,N}\boldsymbol b}{\nu^2}.
\]

To prove the divisibility assertion, fix \(s\mid R\) and compose the expression
for \(G\) with \(\iota_s^*\). By symmetry of the matrix,
\[
        G\circ\iota_s^*
        =
        \frac{(A_{\mathcal X,N}\boldsymbol b)_s}{\nu}\,
        u\circ\Phi_M.
\]
The left-hand side is an integral homomorphism, whereas \(u\circ\Phi_M\) is
primitive by condition \hyperlink{tower-condition-T1}{\(\mathrm{(T1)}\)}. Thus every coordinate of
\(A_{\mathcal X,N}\boldsymbol b\) is divisible by \(\nu\). Multiplication on
the left by \(\boldsymbol b^{T}\) establishes the divisibility in
\eqref{eq:tower-degree-and-first-denominator}.
\end{proof}

Theorem \ref{thm:tower-denominator-criterion} and
\eqref{eq:tower-degree-and-first-denominator} imply the fixed-target
divisibility in Proposition \ref{prop:tower-degree-consequences}.

\begin{proposition}\label{prop:tower-degree-consequences}
Assume conditions \hyperlink{tower-condition-T1}{\(\mathrm{(T1)}\)}--\hyperlink{tower-condition-T3}{\(\mathrm{(T3)}\)}, and let \(N\) be a
multiple of \(M\). Then
every nonconstant morphism \(g:\mathcal X(N)\to E'\) over \(\Q\) satisfies
\begin{equation}\label{eq:tower-fixed-degree}
        \deg(u\circ\pi_M)
        \mid
        \mathcal D_{c_{\mathcal X,u}}(R)\deg g.
\end{equation}
\end{proposition}

\begin{proof}
Translate \(g\) so that \(g(P_N)=0\), and let \(G\) be the associated
homomorphism. Write
\(G=\nu^{-1}\sum_{r\mid R}b_r\Phi^{\mathcal X}_{r,u,N,M}\),
where \(b_r\in\Z\), \(\nu\in\Z_{>0}\),
\(\gcd(\nu,\{b_r\}_{r\mid R})=1\), and
\(\boldsymbol b=(b_r)_{r\mid R}\).
Theorem \ref{thm:tower-denominator-criterion} and the coprimality of
\(\nu\) with the \(b_r\)'s imply
\(\nu\mid\mathcal D_{c_{\mathcal X,u}}(R)\). Equation
\eqref{eq:tower-degree-and-first-denominator} also implies
\(\nu\mid\boldsymbol b^{T}A_{\mathcal X,N}\boldsymbol b\). Hence
\[
        \frac{\mathcal D_{c_{\mathcal X,u}}(R)\deg g}{\deg(u\circ\pi_M)}
        =\frac{\mathcal D_{c_{\mathcal X,u}}(R)}{\nu}
        \frac{\boldsymbol b^{T}A_{\mathcal X,N}\boldsymbol b}{\nu}
        \in\Z,
\]
which is equivalent to \eqref{eq:tower-fixed-degree}.
\end{proof}

If \(c_{\mathcal X,u}=1\), then the old-basis coefficients are
integral, and \(A_{\mathcal X,N}\) determines the degree of the associated
morphism.

\begin{corollary}\label{cor:tower-integral-degree-form}
Assume conditions \hyperlink{tower-condition-T1}{\(\mathrm{(T1)}\)}--\hyperlink{tower-condition-T3}{\(\mathrm{(T3)}\)} and
\(c_{\mathcal X,u}=1\). Let \(N\) be a multiple of \(M\).
For every nonconstant morphism \(g:\mathcal X(N)\to E'\), its
associated homomorphism \(G_g\) based at \(P_N\) has a unique nonzero
coefficient vector \(\boldsymbol b=(b_r)_{r\mid R}\in
\Z^{\{r:\,r\mid R\}}\) satisfying
\[
        G_g=\sum_{r\mid R}b_r\Phi^{\mathcal X}_{r,u,N,M},
        \qquad
        \deg g=\deg(u\circ\pi_M)\,
        \boldsymbol b^{T}A_{\mathcal X,N}\boldsymbol b.
\]
\end{corollary}

\begin{proof}
By Theorem \ref{thm:tower-denominator-criterion}, \(G_g\) has the stated
integral expansion. Its uniqueness follows from \hyperlink{tower-condition-T1}{\(\mathrm{(T1)}\)}. The degree formula is
\eqref{eq:tower-degree-and-first-denominator} with \(\nu=1\).
\end{proof}

For an arbitrary isogenous target \(E''\), composition with a dual
isogeny reduces the denominator estimate to the target \(E_0\).

\begin{theorem}\label{thm:tower-optimal-degree}
Suppose that conditions \hyperlink{tower-condition-T1}{\(\mathrm{(T1)}\)}--\hyperlink{tower-condition-T3}{\(\mathrm{(T3)}\)} hold
with \(E'=E_0\) and \(u=\mathrm{id}_{E_0}\). Let
\(c_{\mathcal X}=c_{\mathcal X,\mathrm{id}_{E_0}}\). Let \(N\) be a multiple
of \(M\). Let \(E''/\Q\) be \(\Q\)-isogenous to \(E_0\). Then every nonconstant
morphism \(g:\mathcal X(N)\to E''\) over \(\Q\) satisfies
\begin{equation}\label{eq:tower-optimal-degree}
        \deg\pi_M\mid\mathcal D_{c_{\mathcal X}}(R)\deg g.
\end{equation}
Moreover,
\begin{equation}\label{eq:sharp-divides-omega}
        \mathcal D_{c_{\mathcal X}}(R)
        \mid c_{\mathcal X}^{\Omega(R)}.
\end{equation}
\end{theorem}

\begin{proof}
Fix a morphism \(g:\mathcal X(N)\to E''\) as in the theorem, and choose a
generator \(v:E_0\to E''\).
Let \(\widehat v:E''\to E_0\) be the dual isogeny.
Postcomposition with \(v\) induces an isomorphism
\[
 \Hom_\Q(J_{\mathcal X}(N),E_0)\otimes_\Z\Q
 \xrightarrow{\ \sim\ }
 \Hom_\Q(J_{\mathcal X}(N),E'')\otimes_\Z\Q,
\]
since
composition with \(\widehat v\) is its inverse up to multiplication by
\(\deg v\). It carries the rational old basis for the target \(E_0\)
to the rational old basis for \(E''\). At level \(M\), Corollary
\ref{cor:optimal-quotient-primitivity}, applied to the optimal quotient
\(\Phi_M\), shows that \(v\circ\Phi_M\) is primitive.
Thus condition \hyperlink{tower-condition-T1}{\(\mathrm{(T1)}\)} holds for the target \(E''\) and
the generator \(v\).
Translate \(g\) so that \(g(P_N)=0\), and let \(G\) be its associated
homomorphism. Write the reduced old-basis expansion as
\[
        G=\frac1\nu\sum_{r\mid R}b_r
        \Phi^{\mathcal X}_{r,v,N,M},
        \qquad
        \gcd(\nu,\{b_r\}_{r\mid R})=1.
\]
Here \(b_r\in\Z\), \(\nu\in\Z_{>0}\), and
\(\boldsymbol b=(b_r)_{r\mid R}\).
By Proposition \ref{prop:tower-degree-matrix},
\(\nu\mid\boldsymbol b^{T}A_{\mathcal X,N}\boldsymbol b\).
Compose \(G\) with \(\widehat v\). Since
\(\widehat v\circ v=[\deg v]\), the reduced denominator in the old basis for
the target \(E_0\) is \(\nu/\gcd(\nu,\deg v)\). Indeed, the coefficients are
\((\deg v)b_r/\nu\), and
\(\gcd(\nu,\{b_r\}_{r\mid R})=1\). Theorem
\ref{thm:tower-denominator-criterion}, applied with target \(E_0\) and
isogeny \(\mathrm{id}_{E_0}\), implies
\(\nu\mid\mathcal D_{c_{\mathcal X}}(R)\deg v\). Then
\[
        \frac{\mathcal D_{c_{\mathcal X}}(R)\deg g}{\deg\pi_M}
        =
        \left(\frac{\mathcal D_{c_{\mathcal X}}(R)\deg v}{\nu}\right)
        \left(\frac{\boldsymbol b^{T}A_{\mathcal X,N}
        \boldsymbol b}{\nu}\right)
        \in\Z.
\]
The first factor on the right is integral by Theorem
\ref{thm:tower-denominator-criterion}, and the second by Proposition
\ref{prop:tower-degree-matrix}.
This establishes \eqref{eq:tower-optimal-degree}. When \(R=1\),
\eqref{eq:sharp-divides-omega} follows from
\(\mathcal D_{c_{\mathcal X}}(1)=1\) and \(\Omega(1)=0\). Suppose \(R>1\).
Fix a prime \(\ell\mid c_{\mathcal X}\), and put
\(a=\ord_\ell(c_{\mathcal X})\) and \(m=\ord_\ell(R)\). If \(m=0\), then
\(a\le a\Omega(R)\), since \(\Omega(R)\ge1\). If \(m\ge1\), then
\(a+\lfloor m/2\rfloor\le am\le a\Omega(R)\).
Taking these inequalities over all primes \(\ell\mid c_{\mathcal X}\) establishes
\eqref{eq:sharp-divides-omega}.
\end{proof}

\subsection{Compatible diamond towers}\label{subsec:compatible-diamond-towers}

For compatible diamond towers, the oldform decomposition with trivial
nebentypus and Corollary \ref{cor:optimal-quotient-primitivity} establish
\hyperlink{tower-condition-T1}{\(\mathrm{(T1)}\)}. The two triangular \(q\)-expansion arguments used for
\hyperlink{tower-condition-T2}{\(\mathrm{(T2)}\)} are carried out on \(\mathcal X_\mu\), with Galois descent for
the Fricke automorphism. Condition \hyperlink{tower-condition-T3}{\(\mathrm{(T3)}\)} is the main departure from
the \(X_0\) argument. An intermediate degeneracy map need not be totally
ramified, so Proposition \ref{prop:connected-degeneracy-kernel} does not apply.
Surjectivity on topological fundamental groups implies surjectivity on
integral homology. Analytic uniformization then shows that the induced
pushforward map on Jacobians is surjective with geometrically connected kernel.
Equation \eqref{eq:tower-map-compatibility} is the composition identity required
in \hyperlink{tower-condition-T3}{\(\mathrm{(T3)}\)}.

For a positive integer \(L\), let \(\Delta_L\) be a subgroup of
\((\Z/L\Z)^\times\) containing \(-1\), and put
\[
        \Gamma_{\Delta_L}(L)=
        \left\{
        \begin{pmatrix}a&b\\c&d\end{pmatrix}\in\Gamma_0(L):
        a\bmod L\in\Delta_L
        \right\}.
\]
Let \(X_{\Delta_L}(L)\) be the smooth projective modular curve over \(\Q\)
associated with \(\Gamma_{\Delta_L}(L)\), and let \(J_{\Delta_L}(L)\) be its
Jacobian. We call a family
\(\boldsymbol\Delta=(\Delta_L)_{M\mid L}\) a
\emph{compatible diamond tower} if, whenever \(L\mid L'\), reduction induces a
surjection
\begin{equation}\label{eq:diamond-surjectivity}
        \Delta_{L'}\twoheadrightarrow\Delta_L.
\end{equation}
The choices \(\Delta_L=(\Z/L\Z)^\times\) and \(\Delta_L=\{\pm1\}\) correspond to
\(X_0(L)\) and the coarse modular curve \(X_1(L)\), respectively. See
\cite{IshiiMomose} for this convention for intermediate modular
curves and their diamond coverings of \(X_0(L)\).
The inclusions
\(\Gamma_1(L)\subseteq\Gamma_{\Delta_L}(L)\subseteq\Gamma_0(L)\) induce the
finite quotient maps
\begin{equation}\label{eq:intermediate-quotient-chain}
        X_1(L)\longrightarrow X_{\Delta_L}(L)\longrightarrow X_0(L).
\end{equation}

The usual degeneracy maps on \(X_1\) are equivariant for the diamond
actions via the reduction homomorphism
\[
        (\Z/N\Z)^\times\longrightarrow(\Z/L\Z)^\times.
\]
Consequently, for \(r\mid N/L\), the corresponding degeneracy map on \(X_1\)
descends over \(\Q\) to the finite morphism
\[
        \iota^{\boldsymbol\Delta}_{r,N,L}:X_{\Delta_N}(N)\longrightarrow
        X_{\Delta_L}(L),
\]
analytically induced by \(\tau\mapsto r\tau\).
Equivalently, conjugation by
\(\left(\begin{smallmatrix}r&0\\0&1\end{smallmatrix}\right)\) maps
\(\Gamma_{\Delta_N}(N)\) into \(\Gamma_{\Delta_L}(L)\).
In this inclusion, the lower-left entry is divisible by \(L\), and the
upper-left entry belongs to \(\Delta_L\) by
\eqref{eq:diamond-surjectivity}. The algebraic degeneracy maps and their
diamond equivariance follow from the modular interpretation
\cite{DiamondShurman}.
The maps are compatible under composition in the sense of
\eqref{eq:tower-map-compatibility}. Write
\(\iota^{(1)}_{r,N,L}\) and \(\iota^{(0)}_{r,N,L}\) for the usual degeneracy
maps on \(X_1\) and \(X_0\), respectively, and use analogous notation at other
levels.

For \(M\mid L\mid N\), \(r\mid N/L\), and \(s\mid L/M\), the following
diagram commutes.
\[
\begin{tikzcd}[column sep=large,row sep=large]
X_1(N)
  \arrow[r,"\iota^{(1)}_{r,N,L}"]
  \arrow[d,two heads]
& X_1(L)
  \arrow[r,"\iota^{(1)}_{s,L,M}"]
  \arrow[d,two heads]
& X_1(M)
  \arrow[d,two heads] \\
X_{\Delta_N}(N)
  \arrow[r,"\iota^{\boldsymbol\Delta}_{r,N,L}"]
  \arrow[d,two heads]
& X_{\Delta_L}(L)
  \arrow[r,"\iota^{\boldsymbol\Delta}_{s,L,M}"]
  \arrow[d,two heads]
& X_{\Delta_M}(M)
  \arrow[d,two heads] \\
X_0(N)
  \arrow[r,"\iota^{(0)}_{r,N,L}"]
& X_0(L)
  \arrow[r,"\iota^{(0)}_{s,L,M}"]
& X_0(M).
\end{tikzcd}
\]
The vertical arrows are the finite surjective quotient maps in
\eqref{eq:intermediate-quotient-chain}. By
\eqref{eq:tower-map-compatibility}, each horizontal composite is the
degeneracy map indexed by \(rs\) from level \(N\) to level \(M\).
For the remainder of this subsection, write
\(\iota_{r,N,L}=\iota^{\boldsymbol\Delta}_{r,N,L}\), and use the same
abbreviation for the induced pullbacks and pushforwards.

For the integral \(q\)-expansion arguments in Lemmas
\ref{lem:intermediate-old-hom} and \ref{lem:intermediate-two-cusp}, we use the
smooth, not necessarily proper, \(\Z\)-curve \(\mathcal X_\mu(L)\). Its
construction combines results of Katz--Mazur and Deligne--Rapoport, as
explained in \cite[Sections~9.3 and~12.3]{DiamondIm}. Its generic fiber is
\(X_1(L)\), and it has a canonical
point \(\infty\in\mathcal X_\mu(L)(\Z)\) whose formal completion is canonically
isomorphic to \(\operatorname{Spf}(\Z[[q]])\). Under the moduli-theoretic
identification \(\mathcal X_\mu(L)_\Q\simeq X_1(L)\), this point corresponds to
the standard analytic cusp \(\infty\), and its \(q\)-adic expansion agrees with
the standard analytic \(q\)-expansion on \(\Gamma_1(L)\)
\cite[\S6.1.2]{ConradEdixhovenStein}. Its image in
\(X_{\Delta_L}(L)\) is a
\(\Q\)-rational cusp, again denoted by \(\infty\). Since \(\tau\mapsto r\tau\)
sends analytic \(\infty\) to \(\infty\), the degeneracy maps in the tower
preserve these chosen cusps. When the level is clear, write
\(j_\infty:X_{\Delta_L}(L)\to J_{\Delta_L}(L)\) for the Abel--Jacobi map based
at this cusp.
Throughout this subsection, \(X_1(L)\) and its intermediate quotients are
understood in the \(\mathcal X_\mu(L)\)-model, under whose dual-isogeny
isomorphism with the usual point-level model the rational cusp denoted by
\(\infty\) corresponds to the cusp \(0\).

Fix a normalized rational newform
\(f(q)=\sum_{n\ge1}a_nq^n\in S_2(\Gamma_0(M))\) of exact level \(M\), which has
trivial nebentypus. Its coefficients
\(a_n\) are integers, since the coefficients of a normalized newform are
algebraic integers \cite{DiamondShurman}. Let
\(\Phi_{\boldsymbol\Delta}:J_{\Delta_M}(M)\twoheadrightarrow
E_{\boldsymbol\Delta}\) be the optimal elliptic quotient attached to \(f\), and let
\(\pi_{\boldsymbol\Delta}:X_{\Delta_M}(M)\to E_{\boldsymbol\Delta}\) be the
parametrization based at \(\infty\). See
\cite{DiamondShurman} for the newform quotient construction. Let
\(E'/\Q\) be in the \(\Q\)-isogeny class of \(E_{\boldsymbol\Delta}\), and choose
a generator \(u:E_{\boldsymbol\Delta}\to E'\). For every
multiple \(N\) of \(M\) and every \(r\mid N/M\), put
\(\Phi^{\boldsymbol\Delta}_{r,u,N,M}
=u\circ\Phi_{\boldsymbol\Delta}\circ\iota_{r,N,M,*}\).
Choose the sign of a minimal N\'eron differential \(\omega_{E'}\) so that
\begin{equation}\label{eq:intermediate-manin}
        (u\circ\pi_{\boldsymbol\Delta})^*\omega_{E'}
        =
        c_{u\circ\pi_{\boldsymbol\Delta}}f(q)\frac{dq}{q},
        \qquad c_{u\circ\pi_{\boldsymbol\Delta}}>0.
\end{equation}

\begin{lemma}\label{lem:intermediate-old-hom}
The scalar \(c_{u\circ\pi_{\boldsymbol\Delta}}\) in
\eqref{eq:intermediate-manin}
is a positive integer. For every multiple \(N\) of \(M\), the maps
\(\{\Phi^{\boldsymbol\Delta}_{r,u,N,M}\}_{r\mid N/M}\) form a \(\Q\)-basis of
\(\Hom_\Q(J_{\Delta_N}(N),E')\otimes_\Z\Q\), and
\(u\circ\Phi_{\boldsymbol\Delta}\) is primitive at level \(M\).
\end{lemma}

\begin{proof}
Let \(S_2(\Gamma_1(N))_f\) denote the \(f\)-isotypic summand.
By the Atkin--Lehner--Li decomposition \cite{Li},
\begin{equation}\label{eq:intermediate-f-oldspace}
        S_2(\Gamma_1(N))_f={}
        \bigoplus_{r\mid N/M}\Q\,f(q^r).
\end{equation}
See also \cite{DiamondShurman} for the oldform multiplicities. Because
\(f\) has trivial nebentypus, every summand in
\eqref{eq:intermediate-f-oldspace} lies in \(S_2(\Gamma_0(N),\Q)\). The
inclusions
\[
 S_2(\Gamma_0(N))\subseteq S_2(\Gamma_{\Delta_N}(N))
 \subseteq S_2(\Gamma_1(N))
\]
therefore show that the \(f\)-isotypic summand for
\(\Gamma_{\Delta_N}(N)\) is precisely the space in
\eqref{eq:intermediate-f-oldspace}.

By Faltings' isogeny theorem \cite{Faltings}, only the \(f\)-isotypic quotient
of \(J_{\Delta_N}(N)\) contributes to
\(\Hom_\Q(J_{\Delta_N}(N),E')\otimes_\Z\Q\). Pullback on differentials by the
\(r\)-th degeneracy map has image the line spanned by \(f(q^r)dq/q\), up to the
nonzero normalization factor \(r\). Under the canonical polarizations, the
dual of this pullback is the pushforward
\(\iota_{r,N,M,*}\)
\cite{ConradEdixhovenStein}. Hence the direct sum in
\eqref{eq:intermediate-f-oldspace} corresponds to the maps
\(\Phi^{\boldsymbol\Delta}_{r,u,N,M}\), which form a \(\Q\)-basis of
\(\Hom_\Q(J_{\Delta_N}(N),E')\otimes_\Z\Q\). At \(N=M\), the basis identity becomes
\[
 \Hom_\Q(J_{\Delta_M}(M),E')\otimes_\Z\Q
 =\Q\,(u\circ\Phi_{\boldsymbol\Delta}).
\]
Since \(\Phi_{\boldsymbol\Delta}\) is optimal, Corollary
\ref{cor:optimal-quotient-primitivity} implies that
\(u\circ\Phi_{\boldsymbol\Delta}\) is primitive.

For the integrality of \(c_{u\circ\pi_{\boldsymbol\Delta}}\), compose
\(u\circ\pi_{\boldsymbol\Delta}\) with
\(\mathcal X_\mu(M)_\Q\to X_{\Delta_M}(M)\). The N\'eron mapping property
\cite{BLR}, applied to the smooth model
\(\mathcal X_\mu(M)\), extends this morphism from the generic fiber of
\(\mathcal X_\mu(M)\) to the N\'eron model of \(E'\). The extended morphism
pulls \(\omega_{E'}\) back to an element of \(\Z[[q]]\,dq\). Equivalently, its coefficient series
relative to \(dq/q\) lies in \(q\Z[[q]]\). The coefficient of \(q\) is
\(c_{u\circ\pi_{\boldsymbol\Delta}}\), since \(a_1=1\). Hence the scalar is an
integer. Its positivity follows from the choice
of sign in \eqref{eq:intermediate-manin}.
\end{proof}

Lemma \ref{lem:intermediate-old-hom} verifies condition
\hyperlink{tower-condition-T1}{\(\mathrm{(T1)}\)} for the compatible diamond tower. Proposition
\ref{prop:tower-degree-matrix} therefore applies with
\(\mathcal X(L)=X_{\Delta_L}(L)\),
\(\Phi_M=\Phi_{\boldsymbol\Delta}\), and
\(\pi_M=\pi_{\boldsymbol\Delta}\). For every multiple \(N\) of \(M\), let
\(A_{\boldsymbol\Delta,N}\) be the matrix indexed by the divisors of \(N/M\)
and characterized by
\begin{equation}\label{eq:intermediate-degree-matrix}
        \Phi_{\boldsymbol\Delta}\circ
        \iota_{r,N,M,*}\circ
        \iota_{s,N,M}^*
        =
        \bigl[(A_{\boldsymbol\Delta,N})_{r,s}\bigr]\circ\Phi_{\boldsymbol\Delta}.
\end{equation}
The matrix \(A_{\boldsymbol\Delta,N}\) is symmetric, positive definite, and
integral, and
\eqref{eq:tower-degree-and-first-denominator} holds with
\(A_{\mathcal X,N}=A_{\boldsymbol\Delta,N}\).

For intermediate modular curves, verification of condition
\hyperlink{tower-condition-T2}{\(\mathrm{(T2)}\)} requires Galois descent because the Fricke automorphism need
not be defined over \(\Q\).

\begin{lemma}\label{lem:intermediate-two-cusp}
Let \(N\) be a multiple of \(M\), let \(\ell\) be a prime, and suppose
\[
        G=\sum_{r\mid R}x_r\Phi^{\boldsymbol\Delta}_{r,u,N,M}
        \in\Hom_\Q(J_{\Delta_N}(N),E')\otimes_\Z\Z_{(\ell)}.
\]
Then
\begin{equation}\label{eq:intermediate-two-cusp}
        \gcd(r,R/r)c_{u\circ\pi_{\boldsymbol\Delta}}x_r\in\Z_{(\ell)}
        \qquad(r\mid R).
\end{equation}
\end{lemma}

\begin{proof}
Choose an integer \(t\) prime to \(\ell\) such that
\(tG\in\Hom_\Q(J_{\Delta_N}(N),E')\). Compose the curve map
\(tG\circ j_\infty:X_{\Delta_N}(N)\to E'\) with
\(\mathcal X_\mu(N)_\Q\to X_{\Delta_N}(N)\). The N\'eron-model argument in the
proof of Lemma \ref{lem:intermediate-old-hom}, applied at level \(N\), shows
that the coefficient series relative to \(dq/q\) lies in \(q\Z[[q]]\). The
triangular argument of Lemma \ref{lem:two-cusp} yields
\begin{equation}\label{eq:intermediate-first-triangular}
        r\,t c_{u\circ\pi_{\boldsymbol\Delta}}x_r\in\Z.
\end{equation}

The argument at the second cusp requires Galois descent. The Fricke matrix
\(w_N=\left(\begin{smallmatrix}0&-1\\N&0\end{smallmatrix}\right)\) normalizes
\(\Gamma_{\Delta_N}(N)\). Indeed, conjugation interchanges the two
diagonal entries modulo \(N\), and these entries are inverse to one another.
Fix a primitive \(N\)-th root of unity \(\zeta_N\).
The induced modular automorphism is defined over
\(K=\Q(\zeta_N)^{\Delta_N}\)
\cite[Lemma~1.8]{JeonKimSchweizer}. For \(X_0(N)\), one has \(K=\Q\), whereas
for a general intermediate curve this equality need not hold. In every case,
\(K/\Q\) is finite Galois.
Put
\(\mathcal W=(tG)_K\circ w_{N,*}:J_{\Delta_N}(N)_K\to E'_K\).
Pullback by a translation fixes invariant differentials.
By the calculation in \eqref{eq:fricke}, the pullback is, up to sign,
\begin{equation}\label{eq:intermediate-fricke-pullback}
        t c_{u\circ\pi_{\boldsymbol\Delta}}
        \sum_{s\mid R}s x_{R/s}f(q^s)\frac{dq}{q}.
\end{equation}
The differential in \eqref{eq:intermediate-fricke-pullback} is defined over
\(\Q\), since \(f\) and the coefficients \(x_r\) are rational.
Pullback along the Abel--Jacobi map identifies invariant
differentials on the Jacobian with regular differentials on the curve
\cite{MilneJacobians}. Thus
\(\mathcal W^*\omega_{E'}\) is fixed by \(\operatorname{Gal}(K/\Q)\). For
every \(\sigma\in\operatorname{Gal}(K/\Q)\), the difference
\({}^{\sigma}\!\mathcal W-\mathcal W\) pulls \(\omega_{E'}\) back to zero.
In characteristic zero, pullback of invariant differentials is injective on
\(\Hom_K(J_{\Delta_N}(N)_K,E'_K)\). Indeed, the image of a nonzero
homomorphism is a positive-dimensional abelian subvariety of \(E'_K\), hence
all of \(E'_K\). In characteristic zero, such a surjective homomorphism is
separable and has nonzero pullback on invariant differentials. Thus
\({}^{\sigma}\!\mathcal W=\mathcal W\) for every
\(\sigma\in\operatorname{Gal}(K/\Q)\).

For abelian varieties \(A,B\) over
\(\Q\), faithfully flat Galois descent
\cite{SGA1} gives the isomorphism
\begin{equation}\label{eq:hom-galois-descent}
        \Hom_\Q(A,B)\xrightarrow{\sim}
        \Hom_K(A_K,B_K)^{\operatorname{Gal}(K/\Q)}.
\end{equation}
By \eqref{eq:hom-galois-descent}, applied with \(A=J_{\Delta_N}(N)\) and
\(B=E'\), the homomorphism \(\mathcal W\) descends uniquely to \(\Q\).

Apply the N\'eron mapping argument used for
\eqref{eq:intermediate-first-triangular} to the curve map
\(\mathcal W\circ j_\infty\) at the rational cusp \(\infty\). Its coefficient
series relative to \(dq/q\) lies in \(q\Z[[q]]\), and the second triangular
argument yields
\begin{equation}\label{eq:intermediate-second-triangular}
        (R/r)t c_{u\circ\pi_{\boldsymbol\Delta}}x_r\in\Z.
\end{equation}
A B\'ezout combination of \eqref{eq:intermediate-first-triangular} and
\eqref{eq:intermediate-second-triangular}, followed by localization at
\(\ell\), proves \eqref{eq:intermediate-two-cusp}.
\end{proof}

\pagebreak[3]
For a general intermediate tower, the repeated-prime step cannot use total
ramification.
If \(N=\ell n\) with \(\ell\mid n\), then
\(\iota_{\ell,N,n}^*q_n=q_N^\ell\), so
the ramification index at \(\infty\) is \(\ell\), but the degree is a
congruence-subgroup index and may be
larger. For example, the standard congruence-subgroup index formula computes
\(\deg(\iota^{(1)}_{2,8,4})=4\)
\cite{DiamondShurman}. Consequently, the total-ramification argument
of Proposition \ref{prop:connected-degeneracy-kernel} does not apply to a
general intermediate tower.

Armstrong's orbit-space theorem \cite{Armstrong} identifies the fundamental
group of a congruence-group quotient of \(\mathfrak H\) with the quotient of
the congruence group by the normal subgroup generated by elliptic stabilizers.
Passing from the quotient of \(\mathfrak H\) to the compact modular curve
amounts to adjoining the cusps
\cite{DiamondShurman}. Van Kampen's theorem for attaching
\(2\)-cells \cite{Hatcher} therefore identifies the
fundamental group of the compactification with the further quotient by the
normal subgroup generated by the parabolic stabilizers. The group-theoretic
assertion needed for surjectivity on fundamental groups is the following.

\begin{lemma}\label{lem:intermediate-parabolic-generation}
Let \(N=\ell n\) with \(\ell\mid n\), and put
\(\alpha=\left(\begin{smallmatrix}\ell&0\\0&1\end{smallmatrix}\right)\) and
\(T=\left(\begin{smallmatrix}1&1\\0&1\end{smallmatrix}\right)\). Set
\(\Gamma=\Gamma_{\Delta_n}(n)\) and
\(\Gamma_\alpha=\alpha\Gamma_{\Delta_N}(N)\alpha^{-1}\).
Then \(\Gamma_\alpha\subseteq\Gamma\) and
\begin{equation}\label{eq:parabolic-generation}
        \langle \Gamma_\alpha,T\rangle=\Gamma.
\end{equation}
\end{lemma}

\begin{proof}
Direct conjugation gives
\[
 \Gamma_\alpha
 =\left\{
 \begin{pmatrix}a&b\\c&d\end{pmatrix}\in\Gamma
 \ \middle|\ \ell\mid b,\quad a\bmod N\in\Delta_N
 \right\}.
\]
Let \(U=\left(\begin{smallmatrix}1&0\\n&1\end{smallmatrix}\right)
\in\Gamma_\alpha\), and fix
\(\beta=\left(\begin{smallmatrix}a&b\\nk&d\end{smallmatrix}\right)\in\Gamma\).
Since \(\ell\mid n\), the integer \(a\) is a unit modulo \(\ell\).
Choose \(e\in\{0,1\}\) so that \(k+ea\) is also a unit modulo \(\ell\).
Then
\[
 U^e\beta=\begin{pmatrix}a&b\\n(k+ea)&d+enb\end{pmatrix}.
\]
By \eqref{eq:diamond-surjectivity}, \(a\bmod n\) lifts to an element
\(\varepsilon\in\Delta_N\). Since \(k+ea\) is a unit modulo \(\ell\), there
is an integer \(t\) such that \(a+tn(k+ea)\equiv\varepsilon\pmod N\).
The upper-left entry of \(T^tU^e\beta\) is therefore congruent to
\(\varepsilon\) modulo \(N\) and is a unit modulo \(\ell\).
Right multiplication by a suitable power \(T^s\) makes the upper-right
entry divisible by \(\ell\) while leaving the upper-left entry unchanged.
We obtain \(T^tU^e\beta T^s\in\Gamma_\alpha\), so
\(\beta\in\langle\Gamma_\alpha,T\rangle\).
The reverse inclusion follows from \(\Gamma_\alpha\subseteq\Gamma\) and
\(T\in\Gamma\).
\end{proof}

Together with the preceding description of the fundamental group,
\eqref{eq:parabolic-generation} implies that the intermediate degeneracy map
is surjective on topological fundamental groups.

\begin{proposition}\label{prop:intermediate-connected-pushforward}
Let \(N=\ell n\) with \(\ell\mid n\). Then
\(\iota_{\ell,N,n,*}:J_{\Delta_N}(N)\to
J_{\Delta_n}(n)\) is surjective and has geometrically connected kernel.
\end{proposition}

\begin{proof}
Retain the notation of Lemma \ref{lem:intermediate-parabolic-generation} and
work over \(\C\). Let \(\overline{\Gamma_\alpha}\subseteq\overline\Gamma\)
be the images of \(\Gamma_\alpha\subseteq\Gamma\) in
\(\mathrm{PSL}_2(\mathbb R)\), and let \(\overline T\) be the image of \(T\).

Armstrong's orbit-space theorem \cite{Armstrong}, applied to the discontinuous
action of \(\overline\Gamma\) on the simply connected upper half-plane
\(\mathfrak H\), identifies
\(\pi_1(\overline\Gamma\backslash\mathfrak H)\) with the quotient of
\(\overline\Gamma\) by the normal subgroup generated by its elliptic stabilizers.
The compact modular curve is obtained by adjoining the cusps
\cite{DiamondShurman}. Adjoining a cusp amounts topologically to
filling in the corresponding puncture with a disk. Van Kampen's theorem for
attaching \(2\)-cells \cite{Hatcher} therefore identifies the
fundamental group of the compactification with the further quotient by the
normal closure of the
images of primitive generators of the parabolic stabilizers. Consequently,
\(\pi_1(X_{\Delta_n}(n)(\C))\simeq \overline\Gamma/\mathcal N\), where
\(\mathcal N\) is the normal subgroup generated by the elliptic and
parabolic stabilizers.

The map of modular curves is induced by
\(\overline{\Gamma_\alpha}\subseteq\overline\Gamma\), so the image of the
induced homomorphism on fundamental groups is
\(\overline{\Gamma_\alpha}\mathcal N/\mathcal N\).
The element \(\overline T\) is a primitive generator of the
stabilizer of the cusp \(\infty\), so it lies in \(\mathcal N\) and has trivial
image in \(\overline\Gamma/\mathcal N\). Since
\(\overline\Gamma=\langle\overline{\Gamma_\alpha},\overline T\rangle\) by
\eqref{eq:parabolic-generation}, one has
\(\overline{\Gamma_\alpha}\mathcal N=\overline\Gamma\). Thus the map
\(\pi_1(X_{\Delta_N}(N)(\C))\to\pi_1(X_{\Delta_n}(n)(\C))\) is surjective.

Set \(X=X_{\Delta_N}(N)(\C)\), \(Y=X_{\Delta_n}(n)(\C)\), and
\(\delta=\iota_{\ell,N,n}^{\mathrm{an}}:X\to Y\).
Since \(H_1(-,\Z)\simeq\pi_1(-)^{\mathrm{ab}}\), the surjective homomorphism
\(\pi_1(X)\to\pi_1(Y)\) induces the surjection
\(\delta_*:H_1(X,\Z)\twoheadrightarrow H_1(Y,\Z)\).
Put \(V_X=H^0(X,\Omega_X^1)^\vee\) and \(\Lambda_X=H_1(X,\Z)\), and put
\(V_Y=H^0(Y,\Omega_Y^1)^\vee\) and \(\Lambda_Y=H_1(Y,\Z)\).
The analytic uniformization of a Jacobian
\cite{BirkenhakeLange}
identifies \(J_{\Delta_N}(N)(\C)\) with \(V_X/\Lambda_X\) and
\(J_{\Delta_n}(n)(\C)\) with \(V_Y/\Lambda_Y\), where the homology lattices
are embedded by integration. The pushforward \(\iota_{\ell,N,n,*}\) is
induced by \(F=(\delta^*)^\vee:V_X\to V_Y\), and functoriality of integration
identifies \(F|_{\Lambda_X}\) with \(\delta_*\). Hence
\(F(\Lambda_X)=\Lambda_Y\). Since
\(\Lambda_Y\) spans the underlying real vector space of \(V_Y\), both \(F\)
and the induced pushforward are surjective. Moreover,
\[
 F^{-1}(\Lambda_Y)=\ker F+\Lambda_X.
\]
The kernel of the induced pushforward is therefore the image of the connected
complex vector space \(\ker F\) in \(V_X/\Lambda_X\), and hence is connected.
Surjectivity and geometric connectedness of the kernel can be checked over
\(\C\), so the proposition follows.

\end{proof}

The compatible diamond tower satisfies \hyperlink{tower-condition-T1}{\(\mathrm{(T1)}\)} by Lemma
\ref{lem:intermediate-old-hom}, \hyperlink{tower-condition-T2}{\(\mathrm{(T2)}\)} by Lemma
\ref{lem:intermediate-two-cusp}, and \hyperlink{tower-condition-T3}{\(\mathrm{(T3)}\)} by Proposition
\ref{prop:intermediate-connected-pushforward} together with
\eqref{eq:tower-map-compatibility}.

\begin{theorem}\label{thm:intermediate-tower}
For every multiple \(N\) of \(M\), one has
\begin{equation}\label{eq:intermediate-lattice}
        \mathcal D_{c_{u\circ\pi_{\boldsymbol\Delta}}}(R)
        \Hom_\Q(J_{\Delta_N}(N),E')
        \subseteq
        \bigoplus_{r\mid R}\Z\Phi^{\boldsymbol\Delta}_{r,u,N,M}.
\end{equation}
If \(c_{u\circ\pi_{\boldsymbol\Delta}}=1\), then the old maps form a
\(\Z\)-basis of
\(\Hom_\Q(J_{\Delta_N}(N),E')\).
\end{theorem}

\begin{proof}
Condition \hyperlink{tower-condition-T1}{\(\mathrm{(T1)}\)} is Lemma \ref{lem:intermediate-old-hom}, and
condition \hyperlink{tower-condition-T2}{\(\mathrm{(T2)}\)}, with
\(c_{\mathcal X,u}=c_{u\circ\pi_{\boldsymbol\Delta}}\), is Lemma
\ref{lem:intermediate-two-cusp}. If \(\ell^2\mid R\),
Proposition \ref{prop:intermediate-connected-pushforward} proves that the
pushforward in \hyperlink{tower-condition-T3}{\(\mathrm{(T3)}\)} is surjective with geometrically connected
kernel. Equation \eqref{eq:tower-map-compatibility} is the remaining
composition identity. Hence conditions \hyperlink{tower-condition-T1}{\(\mathrm{(T1)}\)}--\hyperlink{tower-condition-T3}{\(\mathrm{(T3)}\)}
hold for the fixed target \(E'\).

Theorem \ref{thm:tower-denominator-criterion} establishes
\eqref{eq:intermediate-lattice} and the integral-basis assertion.
\end{proof}

Proposition \ref{prop:tower-degree-consequences} specializes to the
following divisibility on \(X_{\Delta_N}(N)\).

\begin{corollary}\label{cor:intermediate-fixed-degree}
Under the hypotheses and notation of Theorem \ref{thm:intermediate-tower},
every nonconstant morphism \(g:X_{\Delta_N}(N)\to E'\) over \(\Q\) satisfies
\begin{equation}\label{eq:intermediate-fixed-degree}
        \deg(u\circ\pi_{\boldsymbol\Delta})
        \mid
        \mathcal D_{c_{u\circ\pi_{\boldsymbol\Delta}}}(R)\deg g.
\end{equation}
\end{corollary}

\begin{proof}
Apply Proposition \ref{prop:tower-degree-consequences} to the compatible
diamond tower \(\boldsymbol\Delta\).
\end{proof}

When \(c_{u\circ\pi_{\boldsymbol\Delta}}=1\), the integral old basis refines the
divisibility in \eqref{eq:intermediate-fixed-degree} to an exact degree formula
in terms of the quadratic form
\(\boldsymbol b^{T}A_{\boldsymbol\Delta,N}\boldsymbol b\).

\begin{corollary}\label{cor:intermediate-degree-form}
Under the hypotheses and notation of Theorem \ref{thm:intermediate-tower},
suppose that \(c_{u\circ\pi_{\boldsymbol\Delta}}=1\).
For every nonconstant morphism
\(g:X_{\Delta_N}(N)\to E'\), its associated homomorphism \(G_g\) based at
\(\infty\) has a unique nonzero coefficient vector
\(\boldsymbol b=(b_r)_{r\mid R}\in\Z^{\{r:\,r\mid R\}}\) satisfying
\[
        G_g=\sum_{r\mid R}b_r\Phi^{\boldsymbol\Delta}_{r,u,N,M},
        \qquad
        \deg g=\deg(u\circ\pi_{\boldsymbol\Delta})
        \,\boldsymbol b^{T}A_{\boldsymbol\Delta,N}\boldsymbol b.
\]
\end{corollary}

\begin{proof}
Corollary \ref{cor:tower-integral-degree-form} establishes the
integral expansion. The degree formula is
\eqref{eq:tower-degree-and-first-denominator} with
\(A_{\mathcal X,N}=A_{\boldsymbol\Delta,N}\).
\end{proof}

Choose the sign of a minimal N\'eron differential
\(\omega_{E_{\boldsymbol\Delta}}\), and define \(c_{\boldsymbol\Delta}>0\) by
\[
        \pi_{\boldsymbol\Delta}^*\omega_{E_{\boldsymbol\Delta}}
        =c_{\boldsymbol\Delta}f(q)\frac{dq}{q}.
\]
Lemma \ref{lem:intermediate-old-hom}, with
\(u=\mathrm{id}_{E_{\boldsymbol\Delta}}\), implies
\(c_{\boldsymbol\Delta}\in\Z_{>0}\).

\begin{theorem}\label{thm:intermediate-optimal}
Let \(N\) be a multiple of \(M\). For every elliptic curve \(E''/\Q\)
that is \(\Q\)-isogenous to \(E_{\boldsymbol\Delta}\) and every nonconstant morphism
\(g:X_{\Delta_N}(N)\to E''\) over \(\Q\), one has
\begin{equation}\label{eq:intermediate-optimal-divisibility}
        \deg\pi_{\boldsymbol\Delta}
        \mid
        \mathcal D_{c_{\boldsymbol\Delta}}(R)\deg g
        \quad\text{and}\quad
        \deg\pi_{\boldsymbol\Delta}
        \mid
        c_{\boldsymbol\Delta}^{\Omega(R)}\deg g.
\end{equation}
\end{theorem}

\begin{proof}
Condition \hyperlink{tower-condition-T1}{\(\mathrm{(T1)}\)} is Lemma \ref{lem:intermediate-old-hom}. For the
target \(E_{\boldsymbol\Delta}\) and the isogeny
\(\mathrm{id}_{E_{\boldsymbol\Delta}}\), condition \hyperlink{tower-condition-T2}{\(\mathrm{(T2)}\)} is Lemma
\ref{lem:intermediate-two-cusp}, with
\(c_{\mathcal X}=c_{\boldsymbol\Delta}\). Condition \hyperlink{tower-condition-T3}{\(\mathrm{(T3)}\)} consists
of Proposition \ref{prop:intermediate-connected-pushforward} and
\eqref{eq:tower-map-compatibility}. Theorem
\ref{thm:tower-optimal-degree} therefore yields the first divisibility in
\eqref{eq:intermediate-optimal-divisibility} and
\(\mathcal D_{c_{\boldsymbol\Delta}}(R)\mid
c_{\boldsymbol\Delta}^{\Omega(R)}\). The first divisibility in
\eqref{eq:intermediate-optimal-divisibility} therefore implies the second.
\end{proof}

The specialization to \(X_0\) improves the denominator factor obtained in
Section \ref{sec:denominators}. Let \(R>1\), let
\(\ell\mid c\), and put \(a=\ord_\ell(c)\) and \(m=\ord_\ell(R)\). The
\(\ell\)-adic exponents in the factor from \eqref{eq:D-factor} and in
\(\mathcal D_c(R)\) are, respectively,
\(a\max\{1,m\}\) and \(a+\lfloor m/2\rfloor\). One has
\(a+\lfloor m/2\rfloor\le a\max\{1,m\}\), with strict inequality precisely
when \(m\ge3\), or when \(m=2\) and \(a\ge2\). For \(c=1\), both factors are
\(1\).

\begin{corollary}\label{cor:x0-sharp-specialization}
With the notation of Section \ref{sec:denominators}, one has the
\(X_0\)-lattice inclusion
\[
        \mathcal D_{c_{u\circ\pi_E}}(R)\Hom_\Q(J_0(N),E')
        \subseteq
        \bigoplus_{r\mid R}\Z\Phi_{r,u}\, .
\]
Moreover, if \(E''/\Q\) is \(\Q\)-isogenous to \(E\), then every nonconstant
morphism \(g:X_0(N)\to E''\) over \(\Q\) satisfies
\[
        \deg\pi_E\mid \mathcal D_{c_E}(R)\deg g\, .
\]
\end{corollary}

\begin{proof}
Take the full diamond tower \(\Delta_L=(\Z/L\Z)^\times\), for which
\(X_{\Delta_L}(L)=X_0(L)\), and apply Theorems
\ref{thm:intermediate-tower} and \ref{thm:intermediate-optimal}.
\end{proof}

\subsection{The case of \texorpdfstring{\(X_1\)}{X1}}

For each positive integer \(L\), let \(X_1(L)\) be the coarse modular curve and
let \(\infty\in X_1(L)(\Q)\) denote the rational cusp obtained from the canonical
\(\Z\)-point of \(\mathcal X_\mu(L)\) described in Subsection
\ref{subsec:compatible-diamond-towers}.
Write \(J_1(L)=J(X_1(L))\), and retain the notation \(\mathcal D_c\) from
\eqref{eq:sharp-D-definition}. Let \(E_1/\Q\) be the \(X_1(M)\)-optimal elliptic
curve attached to a normalized rational newform
\(f\in S_2(\Gamma_0(M))\) of exact level \(M\).
Let \(\Phi_1:J_1(M)\twoheadrightarrow E_1\) be the optimal quotient, and let
\(\pi_{E_1}:X_1(M)\to E_1\) be the optimal parametrization based at \(\infty\).
Let \(c_1\) be the Manin--Stevens constant of \(\pi_{E_1}\), defined as in
\eqref{eq:intermediate-manin} with \(u=\mathrm{id}_{E_1}\).

Fix an elliptic curve \(E'/\Q\) in the \(\Q\)-isogeny class of \(E_1\) and a
generator \(u:E_1\to E'\). Define \(c_{u\circ\pi_{E_1}}\) as in
\eqref{eq:intermediate-manin}. Both \(c_1\) and \(c_{u\circ\pi_{E_1}}\) are
positive integers by Lemma \ref{lem:intermediate-old-hom}.
Stevens's conjecture \cite[Conjecture I]{StevensStickelberger} can be stated
in terms of \(c_{u\circ\pi_{E_1}}\) as follows.

\begin{conjecture}[Stevens]
\label{conj:stevens-manin}
For every elliptic curve \(E'/\Q\) in the \(\Q\)-isogeny class of \(E_1\)
and every generator \(u\) of \(\Hom_\Q(E_1,E')\), one has
\(c_{u\circ\pi_{E_1}}=1\).
\end{conjecture}

The original formulation asserts the existence of a parametrization
\(\varphi:X_1(M)\to E'\) over \(\Q\) satisfying
\(\varphi^*\omega_{E'}=f(q)dq/q\) for a minimal N\'eron differential
\(\omega_{E'}\). By the level-\(M\) case of Lemma
\ref{lem:intermediate-old-hom}, after translation and a choice of sign,
\(\varphi=[n]\circ u\circ\pi_{E_1}\) for some \(n\in\Z_{>0}\).
Comparison of differentials gives \(1=n c_{u\circ\pi_{E_1}}\), so
\(c_{u\circ\pi_{E_1}}=1\). Conversely, if \(c_{u\circ\pi_{E_1}}=1\), take
\(\varphi=u\circ\pi_{E_1}\).

\v{C}esnavi\v{c}ius \cite{CesnaviciusManinSemistable} proves that
\(\ord_p(c_1)=0\) whenever \(\ord_p(M)\le 1\), so \(c_1=1\) for squarefree
\(M\). By Lemma~2.12 of the same paper, \(c_1=1\) also follows from
\(c_E=1\) for the corresponding \(X_0(M)\)-optimal curve. Cremona's tables
\cite{AgasheRibetStein,CremonaData} therefore imply \(c_1=1\) for
\(M<400000\).

These conclusions concern the optimal target \(E_1\). For an arbitrary
target \(E'\) in its \(\Q\)-isogeny class, the integral-basis result in
Corollary~\ref{cor:x1-stevens-lattice} assumes
Conjecture~\ref{conj:stevens-manin} for \(E'\).

For every multiple \(L\) of \(M\) and every \(r\mid L/M\), write
\(\iota_{r,L,M}:X_1(L)\to X_1(M)\) for the standard degeneracy map
induced analytically by \(\tau\mapsto r\tau\).

\begin{theorem}\label{thm:x1-final}
Let \(N\) be a multiple of \(M\). For every
elliptic curve \(E''/\Q\) in the \(\Q\)-isogeny class of \(E_1\) and every
nonconstant morphism \(g:X_1(N)\to E''\) over \(\Q\), one has
\begin{equation}\label{eq:x1-optimal-divisibility}
        \deg\pi_{E_1}\mid\mathcal D_{c_1}(N/M)\deg g
        \quad\text{and}\quad
        \deg\pi_{E_1}\mid c_1^{\Omega(N/M)}\deg g.
\end{equation}
In particular, if \(c_1=1\), then \(\deg\pi_{E_1}\mid\deg g\).
\end{theorem}

\begin{proof}
For \(\Delta_L=\{\pm1\}\), the compatible diamond tower satisfies
\(X_{\Delta_L}(L)=X_1(L)\) as a coarse modular curve. Hence Theorem
\ref{thm:intermediate-optimal} gives
\eqref{eq:x1-optimal-divisibility}.
\end{proof}

For the fixed target \(E'\), Theorem
\ref{thm:intermediate-tower} specializes to the following \(X_1\)-old-lattice
inclusion.

\begin{theorem}\label{thm:x1-fixed-target}
Let \(N\) be a multiple of \(M\). Then
\[
        \mathcal D_{c_{u\circ\pi_{E_1}}}(N/M)\Hom_\Q(J_1(N),E')
        \subseteq
        \bigoplus_{r\mid N/M}\Z
        \bigl(u\circ\Phi_1\circ\iota_{r,N,M,*}\bigr).
\]
If \(c_{u\circ\pi_{E_1}}=1\), the old maps on the right-hand side form a
\(\Z\)-basis of
\(\Hom_\Q(J_1(N),E')\).
\end{theorem}

\begin{proof}
For the tower \(\Delta_L=\{\pm1\}\) and the identifications
\(X_{\Delta_L}(L)=X_1(L)\), Theorem \ref{thm:intermediate-tower} specializes to
the lattice inclusion and its integral-basis consequence.
\end{proof}

Combining the semistable results of Vatsal
\cite{VatsalMultiplicative} and \v{C}esnavi\v{c}ius
\cite{CesnaviciusManinSemistable} with Theorem
\ref{thm:x1-fixed-target}, we obtain an unconditional \(\Z[1/2]\)-basis
theorem at higher levels.

\begin{theorem}\label{cor:x1-vatsal-lattice}
Assume that \(M\) is squarefree. For every multiple \(N\) of \(M\), the maps
\[
        \bigl\{u\circ\Phi_1\circ\iota_{r,N,M,*}:r\mid N/M\bigr\}
\]
form a \(\Z[1/2]\)-basis of
\(\Hom_\Q(J_1(N),E')\otimes_\Z\Z[1/2]\).
\end{theorem}

\begin{proof}
Let \(E_{\min}\) be the curve of minimal Faltings--Parshin height in the
\(\Q\)-isogeny class of \(E_1\). By Stevens
\cite[Theorem~2.3]{StevensStickelberger}, there are isogenies
\(\alpha:E_{\min}\to E_1\) and \(\beta:E_{\min}\to E'\) that extend to
\'{e}tale morphisms of N\'eron models over \(\Z\). Vatsal
\cite[Theorem~1.10]{VatsalMultiplicative} proves that
\(\deg\alpha\) is a power of \(2\).
Choose a minimal N\'eron differential \(\omega_{E_{\min}}\) and the signs of
\(\alpha\) and \(\beta\) so that
\(\alpha^*\omega_{E_1}=\beta^*\omega_{E'}=\omega_{E_{\min}}\).
Write \(\widehat\alpha:E_1\to E_{\min}\) for the dual isogeny.
Since \(u\) generates \(\Hom_\Q(E_1,E')\), one has
\(\beta\circ\widehat\alpha=[n]\circ u\) for some nonzero \(n\in\Z\).
Moreover, \(c_1=1\) by \v{C}esnavi\v{c}ius
\cite{CesnaviciusManinSemistable}. Hence
\[
        (\deg\alpha)\omega_{E_1}
        =(\beta\circ\widehat\alpha)^*\omega_{E'}
        =n c_{u\circ\pi_{E_1}}\omega_{E_1}.
\]
Thus \(c_{u\circ\pi_{E_1}}\) is a power of \(2\), so
\(\mathcal D_{c_{u\circ\pi_{E_1}}}(N/M)\) is a unit in \(\Z[1/2]\).
After tensoring with \(\Z[1/2]\), the inclusion in Theorem
\ref{thm:x1-fixed-target} is therefore an equality.
\end{proof}

Moreover, Conjecture~\ref{conj:stevens-manin} implies an integral basis at every higher
level, without assuming that \(M\) is squarefree.

\begin{corollary}\label{cor:x1-stevens-lattice}
Assume Conjecture \ref{conj:stevens-manin} for \(E'\). Then, for every multiple
\(N\) of \(M\), the maps
\[
        \bigl\{u\circ\Phi_1\circ\iota_{r,N,M,*}:r\mid N/M\bigr\}
\]
form a \(\Z\)-basis of \(\Hom_\Q(J_1(N),E')\).
\end{corollary}

\begin{proof}
Conjecture~\ref{conj:stevens-manin} gives \(c_{u\circ\pi_{E_1}}=1\), so
Theorem~\ref{thm:x1-fixed-target} applies.
\end{proof}

For each multiple \(N\) of \(M\), let \(A_{1,N}\) be the matrix
\(A_{\boldsymbol\Delta,N}\) defined by
\eqref{eq:intermediate-degree-matrix} for the tower
\(\Delta_L=\{\pm1\}\).
Under Conjecture~\ref{conj:stevens-manin} for \(E'\), the degree is determined exactly by the quadratic
form \(\boldsymbol b^{T}A_{1,N}\boldsymbol b\) associated with the \(\Z\)-basis
in Corollary~\ref{cor:x1-stevens-lattice}.

\begin{corollary}\label{cor:x1-degree-form}
Under the hypotheses and notation of Theorem \ref{thm:x1-fixed-target},
assume Conjecture~\ref{conj:stevens-manin} for \(E'\).
For every nonconstant morphism \(g:X_1(N)\to E'\), its associated
homomorphism \(G_g\) based at \(\infty\) has a unique nonzero coefficient
vector \(\boldsymbol b=(b_r)_{r\mid N/M}\in
\Z^{\{r:\,r\mid N/M\}}\) satisfying
\[
        G_g=\sum_{r\mid N/M}b_r
        \bigl(u\circ\Phi_1\circ\iota_{r,N,M,*}\bigr),
        \qquad
        \deg g=\deg(u\circ\pi_{E_1})\,
        \boldsymbol b^{T}A_{1,N}\boldsymbol b.
\]
\end{corollary}

\begin{proof}
Conjecture~\ref{conj:stevens-manin} gives \(c_{u\circ\pi_{E_1}}=1\).
Apply Corollary \ref{cor:intermediate-degree-form} to the tower
\(\Delta_L=\{\pm1\}\).
\end{proof}


\begin{thebibliography}{99}
\bibitem{AgasheRibetStein} A. Agashe, K. Ribet, and W. Stein, \emph{The Manin constant}, Pure Appl. Math. Q. 2 (2006), no. 2, 617--636.
\bibitem{AgasheRibetSteinCongruence} A. Agashe, K. Ribet, and W. Stein, \emph{The modular degree, congruence primes, and multiplicity one}, in \emph{Number Theory, Analysis and Geometry}, Springer, 2012, 19--49.
\bibitem{Armstrong} M. A. Armstrong, \emph{The fundamental group of the orbit space of a discontinuous group}, Proc. Cambridge Philos. Soc. 64 (1968), 299--301.
\bibitem{AtkinLehner} A. O. L. Atkin and J. Lehner, \emph{Hecke operators on $\Gamma_0(m)$}, Math. Ann. 185 (1970), 134--160.
\bibitem{BirkenhakeLange} C. Birkenhake and H. Lange, \emph{Complex Abelian Varieties}, 2nd ed., Grundlehren der mathematischen Wissenschaften 302, Springer-Verlag, 2004.
\bibitem{BLR} S. Bosch, W. L\"utkebohmert, and M. Raynaud, \emph{N\'eron Models}, Ergebnisse der Mathematik und ihrer Grenzgebiete 21, Springer, 1990.
\bibitem{CesnaviciusManinSemistable} K. \v{C}esnavi\v{c}ius, \emph{The Manin constant in the semistable case}, Compos. Math. 154 (2018), no. 9, 1889--1920.
\bibitem{CesnaviciusNeururerSaha} K. \v{C}esnavi\v{c}ius, M. Neururer, and A. Saha, \emph{The Manin constant and the modular degree}, J. Eur. Math. Soc. 26 (2024), no. 2, 573--637.
\bibitem{ConradEdixhovenStein} B. Conrad, B. Edixhoven, and W. Stein, \emph{$J_1(p)$ has connected fibers}, Doc. Math. 8 (2003), 331--408.
\bibitem{CremonaAlgorithms} J. E. Cremona, \emph{Algorithms for Modular Elliptic Curves}, 2nd ed., Cambridge University Press, 1997.
\bibitem{CremonaData} J. E. Cremona, \emph{Elliptic curve data}, \url{https://johncremona.github.io/ecdata/}, accessed 3 September 2026.
\bibitem{DerickxHwangJeonOrlic} M. Derickx, W. Hwang, D. Jeon, and P. Orli\'c, \emph{Modular curves $X_0(N)$ of density degree $5$}, arXiv:2503.08975v2, 2026.
\bibitem{DerickxOrlic} M. Derickx and P. Orli\'c, \emph{Modular curves $X_0(N)$ with infinitely many quartic points}, Res. Number Theory 10 (2024), Paper No. 42.
\bibitem{DiamondIm} F. Diamond and J. Im, \emph{Modular forms and modular curves}, in \emph{Seminar on Fermat's Last Theorem} (Toronto, ON, 1993--1994), CMS Conf. Proc. 17, American Mathematical Society, 1995, 39--133.
\bibitem{DiamondShurman} F. Diamond and J. Shurman, \emph{A First Course in Modular Forms}, Graduate Texts in Mathematics 228, Springer, 2005.
\bibitem{Faltings} G. Faltings, \emph{Endlichkeitss\"atze f\"ur abelsche Variet\"aten \"uber Zahlk\"orpern}, Invent. Math. 73 (1983), no. 3, 349--366.
\bibitem{GonzalezBielliptic} J. Gonz\'alez, \emph{Equations of bielliptic modular curves}, JP J. Algebra Number Theory Appl. 27 (2012), no. 1, 45--60.
\bibitem{SGA1} A. Grothendieck, \emph{Rev\^etements \'etales et groupe fondamental (SGA~1)}, Lecture Notes in Mathematics 224, Springer-Verlag, 1971.
\bibitem{Hatcher} A. Hatcher, \emph{Algebraic Topology}, Cambridge University Press, 2002.
\bibitem{IshiiMomose} N. Ishii and F. Momose, \emph{Hyperelliptic modular curves}, Tsukuba J. Math. 15 (1991), no. 2, 413--423.
\bibitem{JeonKimSchweizer} D. Jeon, C. H. Kim, and A. Schweizer, \emph{Bielliptic intermediate modular curves}, J. Pure Appl. Algebra 224 (2020), no. 1, 272--299.
\bibitem{KazalickiKohenCorrigendum} M. Kazalicki and D. Kohen, \emph{Corrigendum to ``On a special case of Watkins's conjecture''}, Proc. Amer. Math. Soc. 147 (2019), no. 10, 4563.
\bibitem{Li} W.-C. W. Li, \emph{Newforms and functional equations}, Math. Ann. 212 (1975), 285--315.
\bibitem{MilneJacobians} J. S. Milne, \emph{Jacobian Varieties}, in \emph{Arithmetic Geometry}, G. Cornell and J. H. Silverman (eds.), Springer, 1986, 167--212.
\bibitem{OggHyperelliptic} A. P. Ogg, \emph{Hyperelliptic modular curves}, Bull. Soc. Math. France 102 (1974), 449--462.
\bibitem{Poonen} B. Poonen, \emph{Rational Points on Varieties}, Graduate Studies in Mathematics 186, American Mathematical Society, 2017.
\bibitem{SilvermanAEC} J. H. Silverman, \emph{The Arithmetic of Elliptic Curves}, 2nd ed., Graduate Texts in Mathematics 106, Springer, 2009.
\bibitem{Stein} W. Stein, \emph{Modular Forms, a Computational Approach}, Graduate Studies in Mathematics 79, American Mathematical Society, 2007.
\bibitem{StevensStickelberger} G. Stevens, \emph{Stickelberger elements and modular parametrizations of elliptic curves}, Invent. Math. 98 (1989), no. 1, 75--106.
\bibitem{VatsalMultiplicative} V. Vatsal, \emph{Multiplicative subgroups of $J_0(N)$ and applications to elliptic curves}, J. Inst. Math. Jussieu 4 (2005), no. 2, 281--316.
\bibitem{Watkins} M. Watkins, \emph{Computing the modular degree of an elliptic curve}, Experiment. Math. 11 (2002), no. 4, 487--502.
\end{thebibliography}
\end{document}